\documentclass[reqno]{amsart}
\usepackage[utf8]{inputenc}
\usepackage[T1]{fontenc}
\usepackage{amsthm}
\usepackage{amsmath}
\usepackage{amssymb}
\usepackage{nicefrac}
\usepackage{xcolor}
\usepackage{mathtools} 
\usepackage{color}
\usepackage{hyperref} 
\usepackage{xfrac}
\usepackage[shortlabels]{enumitem}
\hypersetup{
	colorlinks,
	linkcolor={red!50!black},
	citecolor={blue!50!black},
	urlcolor={blue!20!black}
}
\usepackage[nameinlink, capitalise, noabbrev]{cleveref}

\DeclareMathOperator\Sp{\sigma}

\DeclareMathOperator\supp{supp}
\DeclareMathOperator{\tr}{tr}
\DeclareMathOperator{\Aff}{Aff}

\newtheorem{theorem}{Theorem}

\newtheorem{lem}[theorem]{Lemma}
\newtheorem{kor}[theorem]{Corollary}
\newtheorem{prop}[theorem]{Proposition}
\theoremstyle{definition}
\newtheorem{rem}[theorem]{Remark}

\crefname{kor}{Corollary}{Corollaries}
\crefname{lem}{Lemma}{Lemmata}

\numberwithin{equation}{section}
\numberwithin{theorem}{section}
\numberwithin{definition}{section}

\begin{document}
	\title{Self-normalized Sums in Free Probability Theory}
	\author{Leonie Neufeld}
	
	\title{Self-normalized Sums in Free Probability Theory}
	\thanks{Fakult\"at f\"ur Mathematik,
		Universit\"at Bielefeld, 33501 Bielefeld, 
		Germany;  lneufeld@math.uni-bielefeld.de}
	
	\thanks{Funded by the Deutsche Forschungsgemeinschaft (DFG, German Research Foundation) -- IRTG 2235 -- 282638148.}

	
	\begin{abstract}
		We show that the distribution of self-normalized sums of free self-adjoint random variables converges weakly to Wigner's semicircle law under appropriate conditions and estimate the rate of convergence in terms of the Kolmogorov distance. In the case of free identically distributed self-adjoint bounded random variables, we retrieve the standard rate of order $n^{-\nicefrac{1}{2}}$ up to a logarithmic factor, whereas we obtain a rate of order $n^{-\nicefrac{1}{4}}$ in the corresponding unbounded setting. These results provide free versions of certain self-normalized limit theorems in classical probability theory. 
	\end{abstract}	
	
	\maketitle

	\section{Introduction}	 \label{SNS sec introduction}
	Let $(X_i)_{i \in \mathbb{N}}$ be a sequence of free self-adjoint non-commutative random variables and define 
	\begin{align}  \label{def SNS}
		S_n := \sum_{i=1}^{n} X_i, \qquad V_n^2 := \sum_{i=1}^{n} X_i^2, \qquad U_n := 	V_n^{-\nicefrac{1}{2}}S_n V_n^{-\nicefrac{1}{2}}
	\end{align}
	for any $n \in \mathbb{N}.$
	The aim of this work is to prove that the distribution of the \textit{self-normalized sum} $U_n$ 
	converges weakly to Wigner's semicircle law as $n \rightarrow \infty$  and to determine the corresponding rate of convergence. 	
	
	\subsection{Self-normalized sums in classical probability theory} \label{SNS sec: Introduction classical case}
	Before we state our results, let us discuss self-normalized sums in classical probability theory. 
	
	Given a sequence of classical random variables $(X_i)_{i \in \mathbb{N}}$, let $S_n$ and $V_n^2$ be defined as in \eqref{def SNS}. Under the convention $S_n/V_n := 0$ on the event $\{V_n=0\}$, the quotient $\smash{S_n/V_n = 	V_n^{-\nicefrac{1}{2}}S_n V_n^{-\nicefrac{1}{2}}}$  is called a self-normalized sum. The interest in such self-normalized sums is twofold: First, it is known that the random variable $S_n / V_n$ is closely related to the classical Student's $t$-statistic, thus making it relevant in mathematical statistics. Second, self-normalized sums often exhibit an improved limiting behavior and allow limit theorems under weaker distributional assumptions than sums that are normalized in the usual way, i.e.\@ deterministically as in the central limit theorem. The intuition behind this phenomenon is that the denominator $V_n$ weakens irregular fluctuations of the numerator $S_n$ in a stronger way than the usual normalization. 
	
	In the following, we collect some results on the asymptotics of self-normalized sums. We start with an intuitive approach to determine the limiting distribution of $S_n/V_n$ under the exemplary assumption that the random variables $(X_i)_{i \in \mathbb{N}}$ are independent and identically distributed (i.i.d.\@) with mean zero and unit variance. According to the central limit theorem, the sum $n^{-\nicefrac{1}{2}}S_n$ is asymptotically standard normal. The weak law of large numbers states that $n^{-1}V_n^2$ converges in probability to $1$ as $n \rightarrow \infty$. Lastly, an application of Slutsky's theorem yields that $S_n/V_n$ converges in distribution to the standard normal distribution as $n \rightarrow \infty.$
	
	Since the above-mentioned argument considers numerator and denominator separately, the concept of self-normalization hardly comes into effect. In particular, the previously required moment conditions are unnecessarily strong as can be seen by the following self-normalized central limit theorem established by Gin\'{e}, Götze, and Mason \cite{Gine1997}: Given a sequence $(X_i)_{i \in \mathbb{N}}$ of i.i.d.\@ random variables, the self-normalized sum $S_n/V_n$ is asymptotically standard normal if and only if $X_1$ is in the domain of attraction of the normal law and $\mathbb{E}X_1 = 0$. This proves a conjecture made by Logan, Mallows, Rice, and Shepp  \cite{Logan1973}. As already observed by Maller \cite{Maller1981} among others, the if-part of the stated equivalence follows easily by a variation of Raikov's theorem. The proof of the only-if-direction is more involved and makes use of the fact that the moments of $S_n/V_n$ converge to those of the standard normal distribution, whenever $S_n/V_n$ is asymptotically standard normal. It is worth noting that in general the convergence of moments does not hold in the central limit theorem.  
	
	Mason \cite{Mason2005} and Shao \cite{Shao2018} generalized the previously stated self-normalized central limit theorem to independent non-identically distributed (non-i.d.\@) random variables. The rate of convergence to normality in terms of the Kolmogorov distance was analyzed by Bentkus and Götze \cite{Bentkus1996} in the i.i.d.\@ case and by Bentkus, Bloznelis, and Götze \cite{Bentkus1996a} in the independent non-i.d.\@ setting. Under the assumption of zero mean and finite third absolute moments, both works provide a rate of convergence of the distribution of $S_n/V_n$ to the standard normal distribution, which coincides with the corresponding rate in the Berry-Esseen theorem, i.e.\@ with $n^{-\nicefrac{1}{2}}$ or the third Lyapunov fraction. In particular, Lindeberg's condition is sufficient for $S_n/V_n$ to be asymptotically standard normal.  %vgl Petrov, S.118  
	
	We refer to \cite{Pena2009,Shao2013} for extensive overviews of the theory of self-normalized sums, including the analysis of non-uniform Berry-Esseen bounds and self-normalized large deviations as well as the study of all possible (non-Gaussian) limiting distributions of self-normalized sums and generalizations of such sums to the setting of random vectors. 

	\subsection{Self-normalized sums in free probability theory -- Main results} \label{SNS sec: Introduction free case}  We continue by analyzing if and to what extent we can formulate self-normalized limit theorems  in free probability theory. 
	%As already indicated at the beginning of the introduction, our goal is to establish a non-commutative version of the self-normalized central limit theorem with Wigner's semicircle law as the limiting distribution and determine the corresponding rate of convergence. 
	
	Considering the intuitive approach to self-normalized sums presented in \cref{SNS sec: Introduction classical case} and the fact that we have free analogs of the central limit theorem and the law of large numbers, one might think that a free self-normalized limit theorem with Wigner's semicircle law as the limiting distribution can be derived without much effort. Indeed, as we will see in the course of this work, the proof for such a limit theorem is still rather intuitive, but we encounter two difficulties: First, we have to define an appropriate non-commutative analog of a self-normalized sum, with particular attention paid to the normalizing part due to the needed inversion process. 
	Second, we miss a free analog of Slutsky's theorem that helps to combine the limits arising from the free central limit theorem and the law of large numbers.
	
	We solve these two problems as follows: Given a sequence of free self-adjoint non-commutative random variables $(X_i)_{i \in \mathbb{N}}$ and letting the sums $S_n$ and $V_n^2$ be defined as in \eqref{def SNS}, it turns out that $V_n^2$ is invertible under appropriate conditions. Hence, as already indicated at the beginning of this work, the self-adjoint random variable $\smash{U_n = V_n^{-\nicefrac{1}{2}}S_n V_n^{-\nicefrac{1}{2}}}$ from \eqref{def SNS} can be seen as a non-commutative self-normalized sum. The use of Slutsky's theorem as explained in \cref{SNS sec: Introduction classical case} will be substituted by the machinery of Cauchy transforms. For completeness, let us mention that in some cases, Slutsky's theorem can also be replaced by a convergence result for non-commutative rational expressions in strongly convergent random variables; compare to \cref{SNS bounded Remark Yin}. 
	
	A concrete realization of these ideas depends heavily on whether the random variables $(X_i)_{i \in\mathbb{N}}$ are bounded or unbounded. Thus, in the following, we differentiate between \textit{bounded} and \textit{unbounded self-normalized sums}. Recall that random variables are called bounded, whenever they are elements in a $C^*$-algebra belonging to a $C^*$-probability space, whereas unbounded random variables are given by unbounded linear operators on some Hilbert space that are affiliated with a von Neumann algebra. We refer to \cref{SNS sec Preliminaries} for the definitions. \\
	
	Let us begin with the study of bounded self-normalized sums. Our first result shows that self-normalized sums of free self-adjoint (possibly non-i.d.\@) bounded random variables converge to Wigner's semicircle law $\omega$ under appropriate conditions on some Lyapunov-type fraction with a rate that is only slightly worse than the one obtained in the free central limit theorem;  compare to \cref{Berry Esseen CLT bounded}. We measure the rate of convergence in terms of the Kolmogorov distance denoted by $\Delta$ and defined in \eqref{def Kolmogorov distance} and consider two modes of convergence, namely weak convergence and convergence of the moments of the corresponding distributions. 
	
	\begin{theorem} \label{bounded_main_non_id} 
		Let $(\mathcal{A}, \varphi)$ be a $C^*$-probability space with faithful tracial functional $\varphi$ and norm $\Vert \cdot \Vert_{\mathcal{A}}$. Let $(X_i)_{i \in \mathbb{N}}$ be a sequence of free self-adjoint random variables in $\mathcal{A}$ with $\varphi(X_i) = 0$ and $\varphi(X_i^2) = \sigma_i^2$ for all $i \in \mathbb{N}$. For any $n \in \mathbb{N}$, define 
		\begin{align*} 
			S_n := \sum_{i=1}^n X_i, \qquad V_n^2:= \sum_{i=1}^n X_i^2, \qquad B_n^2 := \sum_{i=1}^{n} \sigma_i^2
		\end{align*}
		and assume $B_n^2 >0$. Moreover, set 
		\begin{align*}
			L_{S, 3n} := \frac{\sum_{i=1}^{n} \Vert X_i \Vert_{\mathcal{A}}^3}{B_n^3}, \qquad L_{S, 4n} :=  \frac{\sum_{i=1}^{n} \Vert X_i \Vert_{\mathcal{A}}^4}{B_n^4}.
		\end{align*}
		Then, under the condition $L_{S, 4n} < \nicefrac{1}{16}$, the self-normalized sum $\smash{U_n := V_n^{-\nicefrac{1}{2}}S_n V_n^{-\nicefrac{1}{2}}}$ is well-defined in $\mathcal{A}$. If we additionally assume that $L_{S, 3n}<\nicefrac{1}{2e}$ holds and let $\mu_n$ denote the analytic distribution of $U_n$, we have
		\begin{align*}
			\Delta(\mu_n, \omega) \leq C \max\left\{ \vert \! \log L_{S, 3n} \vert L_{S, 3n}, \vert \! \log L_{S, 4n} \vert L_{S, 4n}^{\nicefrac{1}{2}}\right\}
		\end{align*}
		for some absolute constant $C>0.$ Lastly,  if $\lim_{n \rightarrow \infty }L_{S, 4n} = 0$ holds, we obtain $\mu_n \Rightarrow \omega$ as $n \rightarrow \infty$ and
		\begin{align*}
			\lim_{n \rightarrow \infty} \int_{\mathbb{R}} x^k \mu_n(dx) = \int_{\mathbb{R}} x^k \omega(dx)
		\end{align*} 
		for all $k \in \mathbb{N}$. 
	\end{theorem}
	We will comment on the subscript $S$ (standing for support) in the Lyapunov-type fractions $L_{S, 3n}$ and $L_{S, 4n}$ in \cref{SNS section preliminaries limit theorems}. Moreover, for a discussion of the conditions imposed in the theorem above as well as of possible variations of it, we refer to \cref{bounded comment on proof of main thm non id}. At this point, let us just mention that the assumption $\lim_{n \rightarrow \infty }L_{S, 4n} = 0$ is stronger than the free analog of Lindeberg's condition given in \eqref{Lindeberg condition}. 
	
	As a byproduct of our proof of \cref{bounded_main_non_id}, we obtain that bounded self-normalized sums partly exhibit \textit{superconvergence} to Wigner's semicircle law. More precisely, as stated in the following corollary, the support of the distribution of the self-normalized sum from above is close to that of Wigner's semicircle law. 
	\begin{kor} \label{bounded_main_super}
		Using the notation introduced in \cref{bounded_main_non_id}, assume that $B_n^2 >0$ and $L_{S, 4n}< \nicefrac{1}{64}$ hold. Then, we have
		\begin{align*}
			\supp  \mu_n \subset \left( -2 - \frac{\max_{i \in \{1, \dots, n\}} \Vert X_i \Vert_{\mathcal{A}}}{B_n} - 57 L_{S, 4n}^{\nicefrac{1}{2}},\,2 +  \frac{\max_{ i \in \{1, \dots, n\}} \Vert X_i \Vert_{\mathcal{A}}}{B_n} + 57 L_{S, 4n}^{\nicefrac{1}{2}}\right).
		\end{align*}
	\end{kor}
	Note that the combination of the last two results with \cite[Proposition 2.1]{Collins2014} implies that, under the condition $\lim_{n \rightarrow \infty } L_{S, 4n} = 0$, the self-normalized sum $U_n$ from \cref{bounded_main_non_id} strongly converges to a standard semicircular element as $n \rightarrow \infty$; see \cref{SNS sec preliminaries bounded} for the definitions of strong convergence and semicircular elements.
	
	In the special case of identical distributions, all conditions on the Lyapunov-type fractions made above are satisfied for large $n$. Thus, in that case, \cref{bounded_main_non_id,bounded_main_super} take the following form:
	
	\begin{kor} \label{bounded_main_id}
		Let $(\mathcal{A}, \varphi)$ be a $C^*$-probability space with faithful tracial functional $\varphi$ and norm $\Vert \cdot \Vert_{\mathcal{A}}$. Let $(X_i)_{i \in \mathbb{N}}$  be a sequence of free identically distributed self-adjoint random variables in $\mathcal{A}$ with $\varphi(X_1) = 0$ and $\varphi(X_1^2) = 1$.  For any $n \in \mathbb{N}$, define $S_n$ and $V_n^2$ as in \cref{bounded_main_non_id}.
		Then, for $n > 16\Vert X_1 \Vert _{\mathcal{A}}^4$, the self-normalized sum 
		$\smash{U_n := V_n^{-\nicefrac{1}{2}}S_nV_n^{-\nicefrac{1}{2}}}$
		is well-defined in $\mathcal{A}$. If we let  $\mu_n$ denote the analytic distribution of $U_n$ for those $n$, we have
		\begin{align*}
			\Delta\left(\mu_{n}, \omega\right) \leq C\Vert X_1 \Vert _{\mathcal{A}}^3\frac{\log n}{\sqrt{n}}
		\end{align*}
		for some absolute constant $C>0$. In particular, it follows $\mu_n \Rightarrow \omega$ as $n \rightarrow \infty$ and 
		\begin{align*}
		\lim_{n \rightarrow \infty} \int_{\mathbb{R}} x^k \mu_n(dx) = \int_{\mathbb{R}} x^k \omega(dx) 
		\end{align*} 
		for all $k \in \mathbb{N}$. Lastly, for $n > 64\Vert X_1 \Vert _{\mathcal{A}}^4$, we obtain
		\begin{align*}
			\supp  \mu_n \subset \left( -2 - \frac{58\Vert X_1 \Vert _{\mathcal{A}}^2}{\sqrt{n}},\, 2 +  \frac{58\Vert X_1 \Vert _{\mathcal{A}}^2}{\sqrt{n}} \right).
		\end{align*} 
	\end{kor}
	
	Hence, in the case of identical distributions, the limiting behavior of self-normalized sums is comparable to that of sums considered in the free central limit theorem: First, up to some logarithmic factor, the distributions of both sums exhibit the same rate of convergence to Wigner's semicircle law measured in terms of the Kolmogorov distance; compare to \cref{Berry Esseen CLT bounded}. Second, both sums admit convergence of their support to $[-2,2]$ with identical speed up to constants; see \cite{Kargin2007}. \\

	Let us continue with analyzing unbounded self-normalized sums. We obtain an unbounded analog of \cref{bounded_main_non_id} under the assumption of finite fourth moments and Lindeberg's condition. Below, for a \linebreak $\smash{W^*}$-probability space $(\mathcal{A}, \varphi)$, we denote the corresponding non-commutative  $\mathcal{L}^4$-space by $\mathcal{L}^4(\mathcal{A}, \varphi)$. Recall that $\mathcal{L}^4(\mathcal{A}, \varphi)$ is contained in the algebra $\Aff(\mathcal{A})$ of all operators that are affiliated with $\mathcal{A}$; compare to \cref{SNS sec: preliminaries unbounded} for more details. 
	\begin{theorem} \label{unbounded_main_BE}
		Let ($\mathcal{A}, \varphi$) be a $W^*$-probability space with tracial functional $\varphi$. Let $(X_i)_{i \in \mathbb{N}}$ be a sequence of free self-adjoint random variables in $\mathcal{L}^4(\mathcal{A}, \varphi)$ with $\varphi(X_i) = 0$ and $\varphi(X_i^2) = \sigma_i^2$ for all $i \in \mathbb{N}$.
		For any $n \in \mathbb{N}$, define $S_n, V_n^2$, and $B_n^2$ as in \cref{bounded_main_non_id}. Assume that 
		$B_n^2 >0$ as well as Lindeberg's condition, i.e.\@ 
		\begin{align} \label{Lindeberg condition}
			\forall \varepsilon>0: 	\lim_{n \rightarrow \infty} \frac{1}{B_n^2} \sum_{i=1}^{n} \int_{ \vert x \vert > \varepsilon B_n} x^2 \mu_{X_i}(dx) = 0
		\end{align}
		with $\mu_{X_i}$ being the analytic distribution of $X_i$ for any $i \in \mathbb{N}$, are satisfied. Moreover, 
		set 
			\begin{align*} L_{4n} := \frac{\sum_{i=1}^n \varphi(\vert X_i \vert^4)}{B_n^{4}}.
		\end{align*}
		Then, there exists $n_0 \in \mathbb{N}$ such that the self-normalized sum $\smash{U_n := V_n^{-\nicefrac{1}{2}}S_n V_n^{-\nicefrac{1}{2}}}$ is well-defined in $\Aff(\mathcal{A})$ for all $n \geq n_0$. If we let $\mu_n$ denote the analytic distribution of $U_n$ for those $n$, we have  
		\begin{align*}
			\Delta(\mu_n, \omega) \leq C\left(L_{4n}^{\nicefrac{1}{4}} + \sqrt{n}L_{4n}^{\nicefrac{3}{4}} + nL_{4n}^{\nicefrac{5}{4}}\right)
		\end{align*}
		for some absolute constant $C>0$.  Lastly, if  $\lim_{n \rightarrow \infty} \sqrt{n}L_{4n} = 0$ holds, we obtain $\mu_n \Rightarrow \omega$ as $n \rightarrow \infty.$
	\end{theorem}
	
	In the setting of the above theorem and under the additional condition that the Lyapunov fraction $L_{4n}$ is of order $n^{-1}$, we deduce that $\Delta(\mu_n, \omega)$ is of order $\smash{n^{-\nicefrac{1}{4}}}$ for sufficiently large $n$. This includes the case of identical distributions; the corresponding result is formulated in the following corollary.
	\begin{kor} \label{unbounded main id}
			Let ($\mathcal{A}, \varphi$) be a $W^*$-probability space with tracial functional $\varphi$. Let $(X_i)_{i \in \mathbb{N}}$  be a sequence of free  identically distributed self-adjoint random variables in $\mathcal{L}^4(\mathcal{A}, \varphi)$ with $\varphi(X_1) = 0$ and $\varphi(X_1^2) = 1$. For any $n \in \mathbb{N}$, define $S_n$ and $V_n^2$ as in \cref{bounded_main_non_id}. 
		Then, there exists $n_0 \in \mathbb{N}$ such that the self-normalized sum $\smash{U_n := V_n^{-\nicefrac{1}{2}}S_n V_n^{-\nicefrac{1}{2}}}$ is well-defined in $\Aff(\mathcal{A})$ for all $n \geq n_0$. If we let  $\mu_n$ denote the analytic distribution of $U_n$ for those $n$, 
		we have  
		\begin{align*}
			\Delta(\mu_n, \omega) \leq C \frac{\varphi( \vert X_1 \vert^4)^{\nicefrac{5}{4}}}{n^{\nicefrac{1}{4}}}
		\end{align*}
		for some absolute constant $C>0.$ In particular, it follows $\mu_n \Rightarrow \omega$ as $n \rightarrow \infty.$
	\end{kor}
	
By our approach, we were not able to establish rates of convergence that go beyond the ones obtained in \cref{unbounded_main_BE} and \cref{unbounded main id}. A comment on the difficulties encountered is provided in \cref{SNS unbounded Endbemerkung}. 
	
	For completeness, let us mention that all of the above results apply to self-normalized sums of free self-adjoint random variables contained in a von Neumann algebra. Note that in the  special case of identical distributions, the rate of convergence stated in \cref{bounded_main_id} is better than the one derived in \cref{unbounded main id}. In the non-i.d.\@ case, the condition for the existence of the self-normalized sum given in  \cref{unbounded_main_BE} is more accessible than that in  \cref{bounded_main_non_id}. 
	
	\subsection*{Organization} This paper is structured as follows: We start with recalling the basics of free probability theory and formulate a few auxiliary results in \cref{SNS sec Preliminaries}. In \cref{SNS section preliminaries Bai}, we derive versions of Bai's inequality relating the Kolmogorov distance between two probability measures to their Cauchy transforms. \cref{SNS section bounded} is divided into two parts. In \cref{SNS section bounded proofs}, we prove \cref{bounded_main_non_id,bounded_main_super,bounded_main_id}. Then, in \cref{SNS section bounded example}, we study self-normalized sums of independent GUE matrices, providing an example of how the intuitive approach to self-normalized sums can be applied in the context of random matrices. The results on unbounded self-normalized sums, i.e.\@ \cref{unbounded_main_BE,unbounded main id}, are verified in \cref{SNS sec unbounded}. 
	
	\subsection*{Acknowledgments}  I would like to thank Friedrich Götze for numerous discussions and helpful feedback on this project. Moreover, I thank Franz Lehner for valuable comments.

	\section{Preliminaries} \label{SNS sec Preliminaries}
The aim of this section consists of recalling the basic concepts of free probability theory with main focus on its operator-algebraic aspects. We start with the introduction of non-commutative bounded random variables in the context of $C^*$-probability spaces in \cref{SNS sec preliminaries bounded}, followed by an overview of non-commutative possibly unbounded random variables realized as operators that are affiliated with a von Neumann algebra in \cref{SNS sec: preliminaries unbounded}. In \cref{SNS section preliminaries limit theorems} we discuss rates of convergence in the free central limit theorem in the bounded and unbounded setting.

	\subsection{Bounded random variables} \label{SNS sec preliminaries bounded} 
	In the following, we define $C^*$-probability spaces and introduce non-commutative bounded random variables. We refer to \cite{Nica2006,Kadison1997} for extensive overviews including the definitions and statements given below.
	
	A tuple $(\mathcal{A}, \varphi)$ is a \textit{$*$-probability space} if $\mathcal{A}$ is a unital $*$-algebra over $\mathbb{C}$ with unit $1_{\mathcal{A}} \in \mathcal{A}$ and $\varphi: \mathcal{A} \rightarrow \mathbb{C}$ is a \textit{positive} linear functional with $\varphi(1_{\mathcal{A}}) =1$. Recall that $\varphi$ is said to be positive if we have $\varphi(x^*x) \geq 0$ for all $x \in \mathcal{A}$. For later reference, let us mention that $\varphi$ is called \textit{tracial} if $\varphi(xy) = \varphi(yx)$ holds for all $x,y \in \mathcal{A}$ and \textit{faithful} whenever $\varphi(x^*x) = 0$, $x \in \mathcal{A}$, implies $x=0$. We say that $(\mathcal{A}, \varphi)$ is a \textit{$C^*$-probability space} if it is a $*$-probability space and $\mathcal{A}$ is a unital $C^*$-algebra with  norm $\Vert \cdot \Vert_{\mathcal{A}}: \mathcal{A} \rightarrow [0, \infty)$.
	
	From now on, let $(\mathcal{A}, \varphi)$ be a $C^*$-probability space. Since $\mathcal{A}$ is endowed with an involution, we can define \textit{normal} and \textit{self-adjoint} elements in $\mathcal{A}$ as usual. Let $\Sp(x) := \{ z \in \mathbb{C}: z1_\mathcal{A} - x\,\, \text{is not invertible in $\mathcal{A}$}\}$
	denote the \textit{spectrum} of $x \in \mathcal{A}$. Recall that for any normal $x \in \mathcal{A}$ there exists a \textit{continuous functional calculus} allowing to define $f(x)$ as an element in $\mathcal{A}$ for all continuous functions $f: \Sp(x) \rightarrow \mathbb{C}$. Making use of that  functional calculus, we obtain $\Vert x \Vert_{\mathcal{A}}= \sup \{ \vert z \vert: z \in \Sp(x)\}$ for normal $x \in \mathcal{A}$.
	
	We use the notation $x \geq 0$ to indicate that $x \in \mathcal{A}$ is \textit{positive}, i.e.\@ $x$ is self-adjoint and we have $\Sp(x) \subset [0, \infty)$. It is easy to prove that the set of all positive elements is a convex cone in the real vector space $\mathcal{A}_{\text{sa}}$ of all self-adjoint elements in $\mathcal{A}$. 
	The notion of positivity induces a partial order on $\mathcal{A}_{\text{sa}}$ as follows: We write $x\leq y$ for $x, y \in \mathcal{A}_{\text{sa}}$ if $y-x \geq 0$ holds. Let us recall a few properties of this order: First, for $x,y \in \mathcal{A}_{\text{sa}}$ and $c \in \mathcal{A}$, the relation $x \leq y$ implies $c^*xc \leq c^*yc$. Consequently, for all $x, y \in \mathcal{A}$, we get $0 \leq y^*x^*xy \leq \Vert x \Vert_{\mathcal{A}}^2 \vert y \vert^2$ with $\vert y \vert := (y^*y)^{\nicefrac{1}{2}}$ being defined via functional calculus. Second, for $x,y \in \mathcal{A}_{\text{sa}}$ satisfying $0 \leq x \leq y$, we have $0 \leq x^{\alpha} \leq y^{\alpha}$ for all $\alpha \in (0,1]$, $0 \leq \varphi(x) \leq \varphi(y)$, and $\Vert x \Vert_{\mathcal{A}} \leq \Vert y \Vert_{\mathcal{A}}$. Additionally, if $x$ is invertible in $\mathcal{A}$, then one can prove that $y$ is invertible in $\mathcal{A}$ with $0 \leq y^{-1} \leq x^{-1}$.  %Kad97, Lemma 4.2.8.
	
	We refer to elements in $\mathcal{A}$ as \textit{(bounded) random variables}. A family of unital subalgebras $(\mathcal{A}_i)_{i \in I}$ in $\mathcal{A}$ for some index set $I$ is called \textit{free} if for any choice of indices $i_1, \dots, i_k \in I$, $k \in \mathbb{N}$, with $i_1 \neq i_2$, $i_2 \neq i_3, \dots, i_{k-1} \neq i_k$ and all $x_j \in \mathcal{A}_{i_j}$ with $\varphi(x_{j}) = 0$, $j \in \{1, \dots, k\}$, we have $\varphi(x_1 \cdots x_k) = 0$. A family of
	random variables $(x_i)_{i \in I} \subset \mathcal{A}$ is said to be free if the unital subalgebras generated by $x_i, i \in I$, are free. 
	The \textit{(analytic) distribution $\mu_x$} of a random variable $x \in \mathcal{A}_{\text{sa}}$ is given by the unique probability measure $\mu_x$ on $\mathbb{R}$ with $\supp \mu_x \subset \Sp(x)$ (or even $\supp \mu_x = \Sp(x)$ if $\varphi$ is faithful) satisfying 
	\begin{align*}
		\varphi(f(x)) = \int_{\mathbb{R}} f(t) \mu_x(dt)
	\end{align*}
	for any continuous function $f: \Sp(x) \rightarrow \mathbb{C}$. A random variable in $\mathcal{A}$ with distribution given by Wigner's semicircle law $\omega$ is called a \textit{standard semicircular random variable}.
	Recalling that the \textit{Cauchy transform} $G_\mu$ of a probability measure $\mu$ on $\mathbb{R}$ is given by
	\begin{align} \label{def cauchy transform}
		G_\mu(z) :=  \int_{\mathbb{R}} \frac{1}{z-t} \mu(dt), \qquad z \in \mathbb{C}^+:= \{ z \in \mathbb{C}: \Im z>0\}, 
	\end{align}
	we can write $\varphi\left((z1_{\mathcal{A}}-x)^{-1}\right) = G_{\mu_x}(z)$ for any self-adjoint $x \in \mathcal{A}$ and all $z \in \mathbb{C}^+$. Sometimes, for $z \in \mathbb{C}$, we  abbreviate $z1_{\mathcal{A}}=z$, where it is clear from the context if we consider $z$ as a complex number or an element in $\mathcal{A}.$
	
	Lastly, let us define \textit{strong convergence} of (tuples of) bounded random variables; compare to \cite{Collins2014}. For each $n \in \mathbb{N}$, let $\smash{x^{(n)} = (x_1^{(n)}, \dots, x_m^{(n)})}$ be an $m$-tuple of random variables in $\mathcal{A}$. We say that $(x^{(n)})_{n\in \mathbb{N}} $ strongly converges to an $m$-tuple $x= (x_1, \dots, x_m)$ of elements in $\mathcal{A}$ as $n\rightarrow \infty$ if we have
	\begin{align*} 
		\lim_{n \rightarrow \infty} \varphi\left( P\big(x^{(n)}, (x^{(n)})^*\big)\right) = \varphi(P(x,x^*)) \qquad \text{and} \qquad \lim_{n \rightarrow \infty} \left\Vert P\big(x^{(n)}, (x^{(n)})^* \big) \right \Vert_{\mathcal{A}}  = \Vert P(x,x^*) \Vert_{\mathcal{A}} 
	\end{align*}
for any polynomial $P$ in $2m$ non-commuting formal variables. Note that the above-mentioned convergence in $\varphi$ is also known as \textit{convergence in distribution}.

	\subsection{\texorpdfstring{Unbounded random variables}{Unbounded random variables}} \label{SNS sec: preliminaries unbounded}
	The extension of bounded random variables to unbounded ones is done in the context of unbounded operators on some Hilbert space that are affiliated with a von Neumann algebra acting on that space. We start with introducing the basic setting, followed by a discussion of affiliated operators in \cref{SNS section preliminaries affiliated}. Lastly, in \cref{SNS section preliminaries nc integration theory}, we recall the theory of non-commutative $\mathcal{L}^p$-spaces.
	
	Throughout this section, let $\mathcal{H}$ be a complex Hilbert space with inner product $\langle \cdot, \cdot \rangle_{\mathcal{H}}$ and let $B(\mathcal{H})$ denote the vector space of bounded linear operators on $\mathcal{H}$. We denote the identity operator in $B(\mathcal{H})$ by $1$, whereas the operator norm is given by $\Vert \cdot \Vert$. Moreover, fix a \textit{$W^*$-probability space} $(\mathcal{A}, \varphi)$ with tracial $\varphi$, i.e.\@  $\mathcal{A}$ is a von Neumann algebra acting on $\mathcal{H}$ and $\varphi: \mathcal{A} \rightarrow \mathbb{C}$ is a linear functional with $\varphi(1) =1$, which is faithful, \textit{normal}, tracial, and positive. Recall that $\varphi$ is normal if we have $\sup_{i \in I}\varphi(T_i)  = \varphi(\sup_{i \in I} T_i)$ for each monotone increasing net $(T_i)_{i \in I}$ of operators in $\mathcal{A}$.  %Kadison, Ringrose, Vol II, 7.1.11
	
	Before we discuss affiliated operators, we formulate a few results from the spectral theory of self-adjoint operators on $\mathcal{H}$. 
	For a self-adjoint operator $T: D(T) \rightarrow \mathcal{H}$ with domain $D(T) \subset \mathcal{H}$, let $E_T$ denote its \textit{spectral measure} defined on the Borel $\sigma$-algebra $\mathcal{B}(\mathbb{R})$ taking values in the set of orthogonal projections of $\mathcal{H}$. Recall that the \textit{spectrum} of $T$ is given by $\Sp(T) :=\{ z \in \mathbb{C}: z1 -T \,\, \text{is not a bijection of $D(T)$ onto $\mathcal{H}$}\}$ and that the support of $E_T$ is equal to $\Sp(T)$.  For any $E_T$-almost surely (a.s.\@) finite Borel function $f: \mathbb{R} \rightarrow \mathbb{C} \cup \{ \infty\}$, the closed densely defined 
	operator $f(T): D(f(T)) \rightarrow \mathcal{H}$ is given by the \textit{spectral integral} $f(T) = \int_{\mathbb{R}} f(\lambda) E_T(d\lambda)$ with $D(f(T)) \subset \mathcal{H}$ being some suitably defined domain. The assignment $f \mapsto f(T)$ is known as the \textit{functional calculus} for $T$; we refer to \cite[Chapter 5.3]{Schmuedgen2012} for an extensive overview of properties of that calculus. For $M \in \mathcal{B}(\mathbb{R})$, we call $E_T(M)= \mathbf{1}_M(T)$ a \textit{spectral projection}. 
	
	\subsubsection{Affiliated operators}   \label{SNS section preliminaries affiliated} 
	We say that a closed densely defined operator $T$ on $\mathcal{H}$ with polar decomposition $T = V \vert T \vert $ for $\vert T \vert := (T^*T)^{\nicefrac{1}{2}}$, is \textit{affiliated with $\mathcal{A}$} if we have $V, E_{\vert T \vert}(M) \in \mathcal{A}$ for all $M  \in \mathcal{B}(\mathbb{R})$. Note that a bounded operator $T \in B(\mathcal{H})$ is affiliated with $\mathcal{A}$ if and only if $T \in \mathcal{A}$ holds. %by double commutant theorem
	
	Let $\Aff(\mathcal{A})$ denote the set of all operators on $\mathcal{H}$ which are affiliated with $\mathcal{A}$. Equipped with the operation of taking adjoints, the strong multiplication, and the strong addition (given by the closures of the usual multiplication and addition), $\Aff(\mathcal{A})$ is a $*$-algebra containing $\mathcal{A}$; compare to \cite{Murray1936}.
	From now on, whenever we add or multiply operators in $\Aff(\mathcal{A})$, we always refer to the corresponding strong operation. 
	
   The set of all self-adjoint operators in $\Aff(\mathcal{A})$ is denoted by $\Aff(\mathcal{A})_{\text{sa}}$. As can be shown with the help of \cite[Propositions 4.23 and 5.15]{Schmuedgen2012}, we have $f(T) \in \Aff(\mathcal{A})$ for $T \in \Aff(\mathcal{A})_{\text{sa}}$ and any $E_T$-a.s.\@ finite Borel function $f$. 
	Note that this statement implies that a self-adjoint operator $T$ is affiliated with $\mathcal{A}$ if and only if $f(T) \in \mathcal{A}$ holds for any $E_T$-a.s.\@ bounded Borel function $f$. In particular, for $T \in \Aff(\mathcal{A})_{\text{sa}}$, we get $(z1- T)^{-1} \in \mathcal{A}$ for all $z \not \in \Sp(T)$ and the second resolvent identity remains valid in $\Aff(\mathcal{A})_{\text{sa}}$, i.e.\@ we have
	\begin{align*}
		(z1 -T)^{-1}(T-S)(z1-S)^{-1}  =  (z1-T)^{-1} - (z1-S)^{-1}
	\end{align*}
	for $S,T \in \Aff(\mathcal{A})_{\text{sa}}$ and $z \not \in \Sp(T) \cup \Sp(S)$. Again, for complex $z$, we sometimes abbreviate $z1 = z$.
	%Moreover, the statement implies that the inverse of a self-adjoint invertible operator in $\Aff(\mathcal{A})$ is contained in $\Aff(\mathcal{A})$. We will not go into detail here, but rather refer to \cref{unbounded - V_n^2 invertible lemma} and \cref{unbounded bem invertibility lemma}, where the invertibility of (self-adjoint) operators in $\Aff(\mathcal{A})$ is analyzed in more detail.  %Ref: Nielsen Levy, Bercovi + Pata LLN. Lastly, if $T \in B(\mathcal{H})$, then $T$ is affiliated with $\mathcal{A}$ if and only if $T \in \mathcal{A}$. 
	
Recall that a self-adjoint operator $T:D(T) \rightarrow \mathcal{H}$ is called positive -- denoted by $T \geq 0$ --  if we have $\langle Tx,x \rangle_{\mathcal{H}} \geq 0$ for any $x \in D(T)$. The convex cone of all positive operators in $\Aff(\mathcal{A})_{\text{sa}}$ is denoted by $\Aff(\mathcal{A})_+$. We can define a partial order on $\Aff(\mathcal{A})_{\text{sa}}$ by writing $S \leq T$ for $S,T \in \Aff(\mathcal{A})_{\text{sa}}$ whenever $T-S \geq 0$ holds. Let us state a few properties of the just-defined order; we refer to \cite[Proposition 1]{Dodds2014} and \cite[Proposition 10.14]{Schmuedgen2012} for the corresponding proofs. First, for $S,T \in \Aff(\mathcal{A})_{\text{sa}}$ with $S \leq T$ and $C \in \Aff(\mathcal{A})$, we have $C^*SC \leq C^*TC$. Second, for $S,T \in \Aff(\mathcal{A})_{\text{sa}}$, the inequality $0 \leq S \leq T$ implies $0 \leq S^{\alpha} \leq T^{\alpha}$ for any $\alpha \in (0,1]$. Moreover, if $S$ is invertible in $\Aff(\mathcal{A})$, then $T$ is invertible in $\Aff(\mathcal{A})$ with $0 \leq T^{-1} \leq S^{-1}$. Third, for  $S,T \in \Aff(\mathcal{A})_+$, we have $S \leq T$ if and only if $\smash{D(T^{\nicefrac{1}{2}}) \subset  D(S^{\nicefrac{1}{2}})}$ and $\Vert S^{\nicefrac{1}{2}}x \Vert \leq \Vert T^{\nicefrac{1}{2}}x \Vert$ hold for all $x \in D(T^{\nicefrac{1}{2}})$.

	As introduced in \cite{Bercovici1993}, we call elements in $\smash{\Aff(\mathcal{A})}$ \textit{(unbounded) random variables}. For a random variable $\smash{T \in \Aff(\mathcal{A})}$ with polar decomposition $T = V \vert T \vert$, let $\smash{W^*(T)} \subset \mathcal{A}$ denote the von Neumann subalgebra generated by $V$ and all spectral projections of $\vert T \vert.$ 
	%Observe that $W^*(T)$ is the smallest von Neumann algebra to which $T$ is affiliated, i.e.\@ if $T$ is affiliated to another von Neumann algebra, then this von Neumann algebra contains $W^*(T).$  
A family of random variables $(T_i)_{i \in I}$  in $ \Aff(\mathcal{A})$ for some index set $I$ is said to be \textit{free} if the von Neumann subalgebras $(W^*(T_i))_{i \in I}$ are free in $\mathcal{A}$; compare to \cref{SNS sec preliminaries bounded} for the definition of freeness in $\mathcal{A}.$ %Mingo, Speicher Def. 8.15
	The \textit{(analytic) distribution} $\mu_T$ of an operator $T \in \Aff(\mathcal{A})_{\text{sa}}$ is the unique probability measure on $\mathbb{R}$, concentrated on $\Sp(T)$, satisfying 
	\begin{align*}
		\varphi(f(T)) = \int_{\mathbb{R}} f(t) \mu_T(dt) 
	\end{align*}
	for any ($E_T$-a.s.\@) bounded Borel function $f:\mathbb{R} \rightarrow \mathbb{C}$. As before, the Cauchy transform $G_{\mu_T}$ of $\mu_T$ (as defined in \eqref{def cauchy transform}) can be retrieved via $G_{\mu_T}(z) = \varphi((z1-T)^{-1}) =  \varphi((z-T)^{-1})  $ for $z \in \mathbb{C}^+$. Note that the definitions of freeness and analytic distributions of possibly unbounded random variables are consistent extensions of the corresponding definitions given in the context of bounded random variables. 
	
	For completeness, let us define \textit{free additive convolutions} in the unbounded and bounded setting. Let $X_1, \dots, X_n$ be free self-adjoint (unbounded or bounded) random variables with distributions $\mu_1, \dots, \mu_n$. Then, the distribution of the sum $X_1 + \dots + X_n$ is known as the free additive convolution $\mu_1 \boxplus \dots \boxplus \mu_n.$ We refer to \cite{Mingo2017} and the references given therein for an analytic approach to free additive convolutions.  
	
	\subsubsection{Non-commutative integration theory}  \label{SNS section preliminaries nc integration theory}
	In this section we discuss the non-commutative analog of $\mathcal{L}^p$-spaces. We follow the expositions given in \cite{Fack1986, Hiai2020,Nelson1974}.
	
	The linear functional $\varphi: \mathcal{A} \rightarrow \mathbb{C}$ can be extended to $\Aff(\mathcal{A})_+$ via 
	\begin{align} \label{trace on positive elements}
		\varphi\left( T \right) := \sup_{n \in \mathbb{N}} \varphi\left( \int_0^n \lambda E_T(d\lambda) \right) = \int_0^{\infty} \lambda \varphi\left(E_T(d\lambda)\right) \in [0, \infty], \qquad T \in \Aff(\mathcal{A})_+.
	\end{align}
	Then, the \textit{non-commutative $\mathcal{L}^p$-space on $(\mathcal{A}, \varphi)$} for $0 < p < \infty$ is given by 
	\begin{align*}
		\mathcal{L}^p(\mathcal{A}) = \mathcal{L}^p(\mathcal{A}, \varphi) := \left\{ T \in \Aff(\mathcal{A}) \, \bigg \vert \, \Vert T \Vert_p := \varphi(\vert T \vert^p)^{\nicefrac{1}{p}} < \infty  \right\}.
	\end{align*}
	We set $\mathcal{L}^{\infty}(\mathcal{A}, \varphi) = \mathcal{A}$ and equip $\mathcal{L}^{\infty}(\mathcal{A}, \varphi)$ with the operator norm $\Vert \cdot \Vert$. 
	
	Let us recall a few properties of non-commutative $\mathcal{L}^p$-spaces: Firstly, $(\mathcal{L}^p(\mathcal{A}), \Vert \cdot \Vert_p)$ is a Banach space for $1 \leq p \leq \infty$. Secondly, for $S_1,S_2 \in \mathcal{A}$, $T \in \mathcal{L}^p(\mathcal{A})$, and $0<p \leq \infty$, we have $S_1TS_2 \in \mathcal{L}^p(\mathcal{A})$ with $\Vert S_1TS_2 \Vert_p \leq \Vert S_1 \Vert\Vert S_2 \Vert \Vert T \Vert_p$. Thirdly, Hölder's inequality holds true, i.e.\@ for $0< p,q, r\leq \infty$ with $\smash{r^{-1} = p^{-1} + q^{-1}}$ and for $S \in \mathcal{L}^p(\mathcal{A}), T \in \mathcal{L}^q(\mathcal{A})$, we obtain $ST \in \mathcal{L}^r(\mathcal{A})$ with $ \Vert ST \Vert _r \leq \Vert S \Vert_p \Vert T \Vert_q$. Lastly, we can write
	\begin{align*}
		\Vert T \Vert_p^p = \int_{\mathbb{R}} \vert t \vert^p \mu_T(dt)
	\end{align*}
	for all $T \in \mathcal{L}^p(\mathcal{A}) \cap \Aff(\mathcal{A})_{\text{sa}}$, $0 < p < \infty$, with $\mu_T$ being the analytic distribution of $T.$
	
	For $T \in \Aff(\mathcal{A})$ and $t>0$, the \textit{$t$-th generalized singular number $\mu_t(T)$} is defined by 
	\begin{align*}
		\mu_t(T) := \inf \big\{ s \geq 0: \varphi \big(E_{\vert T \vert}((s, \infty))\big) \leq t \big\}.
	\end{align*}
	%Note that the function $t \mapsto \mu_t(T)$ is non-increasing and right-continuous with $\lim_{t \downarrow 0} \mu_t(T) = \Vert T \Vert \in [0, \infty]$. Moreover, we have $\mu_t(T) = 0$ for all $t \geq 1$, $\mu_t(T) = \mu_t(T^*) = \mu_t(\vert T \vert)$, and $\mu_t(\alpha T) = \vert \alpha\vert \mu_t(T)$ for $\alpha \in \mathbb{C}$. 
	One can prove that $\varphi(T) = \int_{0}^{1} \mu_t(T) dt$ is valid for any $T \in \Aff(\mathcal{A})_+$. Together with $\mu_t(S) \leq \mu_t(T)$ holding for $S,T \in \Aff(\mathcal{A})_+$ with $S \leq T$ and all $t>0$, we get $\varphi(S) \leq \varphi(T)$. Moreover, we have
	\begin{align*}
		\Vert T \Vert_p^p = \int_{0}^1 \mu_t(T)^p dt
	\end{align*}
	for all $T \in \Aff(\mathcal{A})$ and $0 < p< \infty$. 
	
	Since $\mathcal{A} \cap \mathcal{L}^1(\mathcal{A})$ is dense in $\mathcal{L}^1(\mathcal{A})$, the functional $\varphi$ defined on $\mathcal{A}$ can be extended uniquely to a positive linear functional  on $\mathcal{L}^1(\mathcal{A})$ (again denoted by $\varphi$) such that $\vert \varphi(T) \vert \leq \Vert T \Vert_1$ holds for any $T \in \mathcal{L}^1(\mathcal{A})$. Let us recall the following properties of this extension: First, for positive $T \in \mathcal{L}^1(\mathcal{A})$, $\varphi(T)$ coincides with the definition in \eqref{trace on positive elements}. Second, note that $\varphi$ is not tracial in general. However, we can at least prove the following: Let $T \in \Aff(\mathcal{A})$. Then, we have $T^*T \in \mathcal{L}^1(\mathcal{A})$ if and only if  $TT^* \in \mathcal{L}^1(\mathcal{A})$, in which case $\varphi(T^*T) = \varphi(TT^*)$ follows; compare to \cite[Lemma 3.1]{Dodds1993}.
Third, if $S,T \in \mathcal{L}^1(\mathcal{A})$ are free, then we have $ST \in \mathcal{L}^1(\mathcal{A})$ and $\varphi(ST) = \varphi(S) \varphi(T) = \varphi(TS)$; see \cite[Section 3]{Lindsay1997} for a reference.

	\subsection{Rates of convergence in the free CLT}  \label{SNS section preliminaries limit theorems}
	The goal of this section is to recall  rates of convergence in the free central limit theorem (CLT). We refer to \cite{Maejima2023,Nica2006} for several versions of this limit theorem, including a free analog of the Lindeberg CLT in \cite[Theorem 4.1]{Maejima2023}. 
	
Before we start with providing the rates of convergence, let us fix some notation: Given two probability measures $\nu_1, \nu_2$ on $\mathbb{R}$, the \textit{Kolmogorov distance} $\Delta(\nu_1, \nu_2)$ between $\nu_1$ and $\nu_2$ is defined by 
	\begin{align} \label{def Kolmogorov distance}
		\Delta(\nu_1, \nu_2) := \sup_{x \in \mathbb{R}} \vert \nu_1((-\infty, x]) - \nu_2((-\infty, x]) \vert.
	\end{align}
 For a family of (bounded or unbounded) free self-adjoint random variables $X_1, \dots, X_n$ with mean zero, variances $\sigma_1^2, \dots, \sigma_n^2$, and finite $k$-th absolute moments, the \textit{$k$-th Lyapunov fraction $L_{kn}$} is given by
	\begin{align*} 
	 L_{kn} := \frac{\sum_{i=1}^n \varphi( \vert X_i \vert^k)}{B_n^k}, \qquad  B_n^2 := \sum_{i=1}^{n} \sigma_i^2, \qquad k \in \mathbb{N}, k \geq 3.
	\end{align*}
As shown in \cite[Chapter IV, §2, Lemma 2]{Petrov1975}, we have 
\begin{align*}
L_{3n} \leq L_{kn}^{\nicefrac{1}{(k-2)}}
\end{align*}
for $k \geq 3.$
In agreement with the notation used in \cref{bounded_main_non_id}, we define the Lyapunov-type fractions $L_{S, kn}$ as follows: If $X_1, \dots, X_n$ are free self-adjoint bounded random variables with mean zero and variances $\sigma_1^2, \dots, \sigma_n^2$ in some $C^*$-probability space $(\mathcal{A}, \varphi)$ with norm $\Vert \cdot \Vert_{\mathcal{A}}$, we set 
\begin{align*}
	 L_{S, kn} := \frac{\sum_{i=1}^n  \Vert X_i \Vert_{\mathcal{A}}^k}{B_n^k}, \qquad k \in \mathbb{N}, k \geq 3,
\end{align*}
where $B_n^2$ is defined as above. The additional subscript $S$ in $L_{S, kn}$ stands for "support" and is meant to indicate the replacement of the absolute moments considered in $L_{kn}$ by appropriate powers of the norms of the underlying random variables. Recall that the norm of a bounded random variable is closely related to the support of its distribution. It is easy to see that 
\begin{align*}
L_{S, kn}^{\nicefrac{1}{k}} \leq L_{S, 3n}^{\nicefrac{1}{3}}
\end{align*}
holds for all $k \geq 3.$

Now, we continue by recalling rates of convergence in the free CLT for bounded random variables. The following theorem summarizes results taken from \cite{Chistyakov2008,Chistyakov2013,Neufeld2023}. 
	
	\begin{theorem} \label{Berry Esseen CLT bounded}
		Let $(\mathcal{A}, \varphi)$ be a $C^*$-probability space with norm $\Vert \cdot \Vert_{\mathcal{A}}$. Choose free self-adjoint random variables $(X_i)_{i \in \mathbb{N}}$ in $\mathcal{A}$ with $\varphi(X_i) = 0$ and $\varphi(X_i^2) = \sigma_i^2 >0$ for all $i \in \mathbb{N}$. For any $n \in \mathbb{N}$, define
		\begin{align*} 
			B_n^2 := \sum_{i=1}^{n} \sigma_i^2, \qquad 	S_n := \frac{1}{B_n}\sum_{i=1}^n X_i,
		\end{align*}
		and let $\mu_{S_n}$ denote the analytic distribution of $S_n$. Then, we have 
		\begin{align*}
			 \Delta(\mu_{S_n}, \omega) \leq C_0L_{S,3n}  \qquad \text{and} \qquad	\Delta(\mu_{S_n}, \omega) \leq C_1L_{3n}^{\nicefrac{1}{2}} 
		\end{align*}
		for constants $C_0, C_1 >0.$
		In the special case of identical distributions with $\varphi(X_1^2) =1$, we obtain
		\begin{align*}
			\Delta(\mu_{S_n}, \omega) \leq C_2\frac{\varphi( \vert X_1 \vert^3)}{\sqrt{n}}
		\end{align*} 
		for some $C_2>0$.
	\end{theorem}

	The case of (possibly) unbounded random variables was studied in \cite{Chistyakov2008, Chistyakov2013}. The corresponding results read as follows:
	\begin{theorem} \label{Berry Esseen CLT unbounded}
		Let  $(\mathcal{A}, \varphi)$ be a $W^*$-probability space with tracial functional $\varphi$. Choose free self-adjoint random variables $(X_i)_{i \in \mathbb{N}}$ in $\mathcal{L}^3(\mathcal{A}) \subset \Aff(\mathcal{A})$ with $\varphi(X_i) = 0$ and $\varphi(X_i^2) = \sigma_i^2 >0$ for all $i \in \mathbb{N}$. For any $n \in \mathbb{N}$, define $B_n^2, S_n$, and $\mu_{S_n}$ analogously to \cref{Berry Esseen CLT bounded}. Then, we have 
		\begin{align*} 
			\Delta(\mu_{S_n}, \omega) \leq C_1L_{3n}^{\nicefrac{1}{2}}
		\end{align*}
		for a constant $C_1>0$. In the special case of identical distributions with $\varphi(X_1^2) = 1$, we obtain
		\begin{align*}
			\Delta(\mu_{S_n}, \omega) \leq C_2\frac{ \varphi(\vert X_1 \vert^3)}{\sqrt{n}}
		\end{align*} 
		for some $C_2>0$. 
	\end{theorem}
	
	For completeness, let us mention that throughout this work, rates of convergence in the free law of large numbers (LLN; compare to \cite{Bercovici1996, Lindsay1997} for references) will be of relevance too. Since those can be derived easily, we do not formulate them here. Instead, we just refer to \cref{bounded_Vn is invertible} and \eqref{LLN unbounded rate of convergence example}, where we calculate the corresponding rates exemplarily.
	
 	\section{Variations of Bai's inequality} \label{SNS section preliminaries Bai}
	
	In this section we prove two inequalities relating the Kolmogorov distance between two probability measures to their Cauchy transforms. A first contribution in that direction is due to Bai \cite[Section 2]{Bai1993}. Making use of Cauchy's integral theorem, Götze and Tikhomirov \cite[Corollary 2.3]{Goetze2003} established a version of Bai's inequality. The following theorem is a generalization of their result. 
	
	\begin{theorem} \label{Bai Goetze Version more general}
		Let $\mu$ and $\nu$ be probability measures on $\mathbb{R}$ with distribution functions $F_\mu$ and $F_\nu$ and Cauchy transforms $G_\mu$ and $G_\nu$. Assume that 
		\begin{align} \label{Bai Bed Integral endlich}	
			\int_{-\infty}^{\infty} \vert F_\mu(x) - F_\nu(x) \vert dx < \infty
		\end{align}
		holds. Choose $v\in (0,1), \varepsilon \in (0,2)$, and $a, \gamma >0$ in such a way that
		\begin{align*}
			\gamma = \frac{1}{\pi} \int_{\vert x \vert < a} \frac{1}{1+x^2} dx > \frac{1}{2} \qquad \text{and} \qquad \varepsilon>2va
		\end{align*}
		are satisfied. Then, we have 
		\begin{align*}
			\!\!\!	\Delta(\mu, \nu) \leq &\, C_\gamma \left(  \int_{-\infty}^{2} \!\vert G_\mu(u+i) - G_\nu(u+i) \vert du + \sup_{x \in [-2+\nicefrac{\varepsilon}{2}, 2-\nicefrac{\varepsilon}{2}]} \int_{v}^{1} \! \vert G_\mu(x+iy) - G_\nu(x+iy) \vert dy \right.\\ &  \left. \qquad \,\,\, + \frac{1}{v} \sup_{x \in \mathbb{R}} \int_{\vert y \vert < 2va} \vert F_\nu(x) - F_\nu(x+y) \vert dy +  \gamma \pi \max\{ F_\nu(-2+\varepsilon), 1-F_\nu(2-\varepsilon)\} \right), 
		\end{align*}
		where $C_\gamma>0$ is given by $C_\gamma := ((2\gamma -1)\pi)^{-1}.$
		\begin{proof}  
			Let $v, \varepsilon, a$, and $\gamma$ be as given above and define 
			\begin{align*}
				I_\varepsilon :=  [-2+\varepsilon, 2-\varepsilon], \qquad I_\varepsilon' := [-2+\nicefrac{\varepsilon}{2}, 2 - \nicefrac{\varepsilon}{2}], \qquad \Delta_\varepsilon(\mu, \nu) := \sup_{x \in I_\varepsilon} \vert F_{\mu}(x) - F_\nu(x) \vert.
			\end{align*}
			Following some of the ideas presented in \cite{Bai1993,Goetze2003}, we proceed in several steps: We start with establishing two preliminary bounds on $\Delta(\mu, \nu)$; see \eqref{estimate Delta and Delta_epsilon} and \eqref{Bai Götze Version intermediate ineq} below. The first one shows that it suffices to analyze $\Delta_\varepsilon(\mu, \nu)$, whereas the second one provides an estimate for $\Delta_\varepsilon(\mu, \nu)$. Lastly, using Cauchy's integral theorem, we obtain the claimed inequality.
			
			Let us begin by relating $\Delta(\mu, \nu)$ to $\Delta_\varepsilon(\mu, \nu)$. Observe that 
			\begin{align*}
				\Delta(\mu, \nu) = \max\left\{ \Delta_\varepsilon(\mu, \nu), \sup_{x > 2-\varepsilon} \vert F_\mu(x) - F_\nu(x) \vert, \sup_{x<-2+\varepsilon} \vert F_\mu(x) -F_\nu(x) \vert \right\}
			\end{align*}
			holds. In order to bound the term $\sup_{x > 2-\varepsilon} \vert F_\mu(x) - F_\nu(x) \vert$, note that we have 
			\begin{align*}
				1-F_\nu(2-\varepsilon) \geq 1-F_\nu(x) &\geq F_\mu(x) - F_\nu(x) \\ &\geq F_\mu(2-\varepsilon)  - F_\nu(2-\varepsilon) -1 + F_\nu(2-\varepsilon) \geq - \Delta_\varepsilon(\mu, \nu)- 1 + F_\nu(2-\varepsilon)
			\end{align*}
			for any $x > 2-\varepsilon$.
			This implies 
			\begin{align*}
				\sup_{x > 2-\varepsilon}	\vert F_\mu(x) - F_\nu(x) \vert \leq \max\{ 1- F_\nu(2-\varepsilon), \Delta_\varepsilon(\mu, \nu) +1 - F_\nu(2-\varepsilon) \} = \Delta_\varepsilon(\mu, \nu) +1 - F_\nu(2-\varepsilon).
			\end{align*}
			Similarly, for any $x < -2+\varepsilon$, we get
			\begin{align*}
				-F_\nu(-2+\varepsilon) \leq F_\mu(x) - F_\nu(-2+\varepsilon) &\leq F_\mu(x) - F_\nu(x)  \\ & \leq F_\mu(-2+\varepsilon) - F_\nu(-2+\varepsilon) + F_\nu(-2+\varepsilon) \leq \Delta_\varepsilon(\mu, \nu) + F_\nu(-2+\varepsilon)
			\end{align*}
			yielding the estimate
			\begin{align*}
				\sup_{x < -2+\varepsilon}	\vert F_\mu(x) - F_\nu(x) \vert \leq  \Delta_\varepsilon(\mu, \nu) +  F_\nu(-2+\varepsilon).
			\end{align*}
			Hence, we deduce
			\begin{align} \label{estimate Delta and Delta_epsilon}
				\Delta(\mu, \nu) \leq \Delta_\varepsilon(\mu, \nu) + \max\{F_\nu(-2+\varepsilon), 1-F_\nu(2-\varepsilon)\}.
			\end{align}
			
			We continue by bounding the quantity $\Delta_\varepsilon(\mu, \nu)$. Due to \eqref{estimate Delta and Delta_epsilon} and $\gamma(2\gamma -1)^{-1} \geq 1$, we are allowed to assume that $\Delta_\varepsilon(\mu, \nu) >0$ holds. 
			By definition of $\Delta_\varepsilon(\mu, \nu)$, there exists a sequence $(t_n)_{n \in \mathbb{N}} \subset I_\varepsilon$ with $\lim_{n \rightarrow \infty} \vert F_\mu(t_n)-F_\nu(t_n)  \vert=\Delta_\varepsilon(\mu, \nu)$. After passing to a convergent subsequence if necessary, it follows $\lim_{n \rightarrow \infty} F_\mu(t_{n})-F_\nu(t_{n}) \in \{\Delta_\varepsilon(\mu, \nu),-\Delta_\varepsilon(\mu, \nu)\}$. 
			
			Let us start with considering the case $\lim_{n \rightarrow \infty} F_\mu(t_n)-F_\nu(t_n) = \Delta_\varepsilon(\mu, \nu)$. By integration by parts for the Riemann-Stieltjes integral, we have
			\begin{align*}
				\Im G_{\mu}(u+iv) =	\int_{-\infty}^{\infty} \frac{v}{(y-u)^2 + v^2} \mu(dy) = - \int_{-\infty}^{\infty} \frac{2v(y-u)F_\mu(y)}{((y-u)^2 + v^2)^2} dy, \qquad u \in \mathbb{R}.
			\end{align*}
			Clearly, a similar equation holds for $\nu$. For any $x \in \mathbb{R}$, we obtain
			\begin{align*}
				\left \vert  \int_{-\infty}^x \Im(G_\mu(u+iv) 
				- G_\nu(u+iv)) du \right \vert
				% &  \geq \int_{-\infty}^x \left( \int_{-\infty}^\infty  \frac{v}{(y-u)^2 + v^2}\nu(dy) - \int_{-\infty}^\infty  \frac{v}{(y-u)^2 + v^2}\mu(dy)\right)du \\
				&\geq \int_{-\infty}^x \left( \int_{-\infty}^\infty  \frac{2v(y-u)(F_\mu(y) - F_\nu(y))}{((y-u)^2 + v^2)^2}dy \right)du \\
				&= \int_{-\infty}^\infty (F_\mu(y) - F_\nu(y))
				\left( \int_{-\infty}^x \frac{2v(y-u)}{((y-u)^2 + v^2)^2}du\right)dy \\
				& = \int_{-\infty}^\infty 
				\frac{(F_\mu(y) - F_\nu(y))v}{v^2+(x-y)^2}dy \\
				& = \int_{-\infty}^\infty \frac{F_\mu(x-vy) - F_\nu(x-vy)}{y^2+1} dy.
			\end{align*}
		Note that we applied the Fubini-Tonelli theorem in the first equality, which is possible due to \eqref{Bai Bed Integral endlich}.
			%Note that the function $(u,y) \mapsto  \frac{2v(y-u)(F_\nu(y) - F_\mu(y))}{((y-u)^2 + v^2)^2}$ is Borel- measurable. Moreover, we have 
			%		\begin{align*}
				%			 \int_{-\infty}^\infty (F_\nu(y) - F_\mu(y))\left( \int_{-\infty}^x \frac{2v(y-u)}{((y-u)^2 + v^2)^2}du\right)dy  \leq $\frac{2}{v}$ \int_{-\infty}^\infty (F_\nu(y) - F_\mu(y))dy  < infty,
				%		\end{align*}
			%das reicht aus, eines der iterierten integrale muss existieren
			The monotonicity of $F_\mu$ implies
			\begin{align*}
				\int_{\vert y \vert < a} \!\! \frac{F_\mu(x-vy) - F_\nu(x-vy)}{y^2+1} dy 
				& \geq  \gamma\pi(F_\mu(x-va) - F_\nu(x-va)) - \int_{\vert y \vert < a} \!\! \vert F_\nu(x-vy) - F_\nu(x-va) \vert dy \\ & = \gamma\pi(F_\mu(x-va) - F_\nu(x-va)) -  \frac{1}{v}\int_{\vert y \vert < va} \!\!\vert F_\nu(x-y) - F_\nu(x-va) \vert dy
			\end{align*}
			for $x$ as above.
			By observing that 
			\begin{align*}
				\int_{\vert y \vert \geq a} \frac{F_\mu(x-vy) - F_\nu(x-vy)}{y^2+1} dy \geq -\Delta(\mu, \nu)  \int_{\vert y \vert \geq a} \frac{1}{y^2+1} dy = -\Delta(\mu, \nu)(1-\gamma)\pi
			\end{align*}
			and $t_n + va \in I_\varepsilon'$ hold, we get
			\begin{align*}
				\sup_{x \in I_\varepsilon'} & \, \left \vert\int_{-\infty}^x \Im(G_\mu(u+iv) - G_\nu(u+iv))  du \right \vert \\  \geq  & \, \left \vert  \int_{-\infty}^{t_n + va} \Im(G_\mu(u+iv) 
				- G_\nu(u+iv)) du \right \vert \\ \geq & \,  \gamma \pi(F_\mu(t_n)- F_\nu(t_n))   -\frac{1}{v} \int_{\vert y \vert < va} \vert F_\nu(t_n) - F_\nu(t_n + va -y) \vert dy  - (1-\gamma)\pi\Delta(\mu, \nu)
			\end{align*}
			for any $n \in \mathbb{N}$.
		Using integration by substitution, it follows
			\begin{align*}
				\int_{\vert y \vert < va} \!\! \vert F_\nu(t_n) - F_\nu(t_n + va -y) \vert dy  & = \!\int_{0}^{2va} \!\! \vert F_\nu(t_n) - F_\nu(t_n + y) \vert dy \leq \sup_{x \in \mathbb{R}}\int_{\vert y \vert < 2va}  \!\!  \vert F_\nu(x) - F_\nu(x+y) \vert dy.
			\end{align*}
			Combining the last two inequalities and passing to the limit $n \rightarrow \infty$, we arrive at
			\begin{align*}
				\sup_{x \in I_\varepsilon'}	& \, \left \vert  \int_{-\infty}^x \Im(G_\mu(u+iv) 
				- G_\nu(u+iv)) du \right \vert \\ \geq  & \, \gamma \pi\Delta_\varepsilon(\mu, \nu)   -\frac{1}{v} \sup_{x \in \mathbb{R}}\int_{\vert y \vert < 2va} \!\! \vert F_\nu(x) - F_\nu(x+y) \vert dy - (1-\gamma)\pi\Delta(\mu, \nu). 
			\end{align*}
			Together with \eqref{estimate Delta and Delta_epsilon}, we deduce 
			\begin{align*}
				\sup_{x \in I_\varepsilon'} &\, \left \vert \int_{-\infty}^x \Im(G_\mu(u+iv) 
				- G_\nu(u+iv)) du \right \vert \\ \geq &\, (2\gamma -1) \pi \Delta_\varepsilon(\mu, \nu)- (1-\gamma)\pi \max\{F_\nu(-2+\varepsilon), 1-F_\nu(2-\varepsilon)\} \\ &\,\,\,\, - \frac{1}{v}\sup_{x \in \mathbb{R}}\int_{\vert y \vert < 2va} \!\! \vert F_\nu(x) - F_\nu(x+y) \vert dy.
			\end{align*}
			
			In the second case, i.e.\@ if $\lim_{n \rightarrow \infty} F_\mu(t_n) - F_\nu(t_n) =-\Delta_\varepsilon(\mu, \nu)$ holds, we switch the roles of $F_\mu$ and $F_\nu$ and obtain the same inequality as above. 
			In total, we have proved
			\begin{align} \label{Bai Götze Version intermediate ineq}
				\begin{split}
					\Delta_\varepsilon(\mu, \nu) &\leq \frac{1}{(2\gamma -1)\pi} \sup_{x \in I_\varepsilon'}  \bigg \vert\int_{-\infty}^x \Im(G_\mu(u+iv) - G_\nu(u+iv))  du \bigg \vert \\ & \qquad + \frac{1}{( 2\gamma -1)\pi}  \frac{1}{v}\sup_{x\in \mathbb{R}}\int_{\vert y \vert < 2va} \vert F_\nu(x) - F_\nu(x+y) \vert dy   \\ & \qquad \, + \frac{1-\gamma}{2\gamma -1} \max\{F_\nu(-2+\varepsilon), 1-F_\nu(2-\varepsilon)\}.
				\end{split}
			\end{align}
			
			Lastly, we apply Cauchy's integral theorem for rectangular contours to the first integral on the right-hand side in \eqref{Bai Götze Version intermediate ineq}. For any fixed $x \in I_\varepsilon'$ and $M>0$ with $x > -M$, it follows
			\begin{align*} 
				\begin{split}
					\int_{-M}^x \!\! G_\mu(u+iv) - G_\nu(u+iv)  du  = \!&\int_{-M}^x  \!\! G_\mu(u+i)  - G_\nu(u+i) du  
					+ i\int_{v}^{1} \! \! G_\mu(-M+iy) - G_\nu(-M+iy) dy \\ & \qquad - i
					\int_{v}^{1} \!\! G_\mu(x+iy) - G_\nu(x+iy)  dy.
				\end{split}
			\end{align*}
			Observe that 
			\begin{align*}
				\vert G_\mu(-M + iy) \vert \leq \int_{\vert t \vert \leq \nicefrac{M}{2}} \frac{1}{\vert -M + iy -t \vert} \mu(dt) +\frac{1}{y} \int_{\vert t \vert > \nicefrac{M}{2}} \mu(dt) \leq \frac{2}{M} + \frac{\mu ( ( \nicefrac{M}{2}, \infty))}{y}
			\end{align*}
			holds for any $y>0.$ An analog inequality is valid for $G_\nu(-M + iy).$ 
			Thus, we have
			\begin{align*}
				\left \vert \int_{v}^{1}  G_\mu(-M+iy) - G_\nu(-M+iy) dy \right \vert \leq \frac{4}{M} +  \big(\mu ( ( \nicefrac{M}{2}, \infty)) + \nu ( ( \nicefrac{M}{2}, \infty)) \big) \vert \! \log(v) \vert \rightarrow 0 
			\end{align*}
			as $M \rightarrow \infty$. By the dominated convergence theorem, it follows
			% Since $u \mapsto \Im G_\mu(u+iv)$ is negative and we have
			%		\begin{align*}
				%			\int_{-\infty}^{\infty} \vert \Im G_\mu(u+iv) \vert du = \pi,
				%		\end{align*}
			%		it follows
			\begin{align*}
				\int_{-\infty}^x \! \Im \left( G_\mu(u+ib) - G_{\nu}(u+ib) \right)   du 	=  \lim_{M \rightarrow \infty} \Im \int_{-M}^x  G_\mu(u+ib) - G_\nu(u+ib)du, \qquad b \in \{v,1\}
			\end{align*}
			leading to 
			\begin{align*}
				\int_{-\infty}^x \! \Im (G_\mu(u+iv) - G_\nu(u+iv))  du  = & \int_{-\infty}^x  \! \Im (G_\mu(u+i)  - G_\nu(u+i)) du  \\ & \qquad 
				- i	\int_{v}^{1} \! \Im (G_\mu(x+iy) - G_\nu(x+iy))  dy
			\end{align*}
			for all $x \in I_\varepsilon'$. This implies 
			\begin{align*} 
					\sup_{x \in I_\varepsilon'} \bigg \vert\int_{-\infty}^x \Im(G_\mu(u+iv) - G_\nu(u+iv))  du \bigg \vert & \leq  \int_{-\infty}^2 \vert G_\mu(u+i) - G_\nu(u+i) \vert du \\ & \qquad + \sup_{x \in I_\varepsilon'} \int_{v}^{1}  \vert G_\mu(x+iy) - G_\nu(x+iy) \vert dy. 
			\end{align*}
			Combining \eqref{estimate Delta and Delta_epsilon} and \eqref{Bai Götze Version intermediate ineq} with the inequality above, the claim follows.
		\end{proof}
	\end{theorem}
	
	The result of the following corollary can be obtained easily from \cref{Bai Goetze Version more general}. The idea is to replace the domain of integration $(-\infty, 2)$ of the first integral in the upper bound for the Kolmogorov distance as given in \cref{Bai Goetze Version more general} by a bounded interval at the cost of gaining an additional term in that upper bound. The proof given below is inspired by the proof of \cite[Corollary 2.3]{Bai1993}.
	\begin{kor} \label{Bai Goetze more general endl. Intgrenzen}
		In the setting of Theorem \ref{Bai Goetze Version more general}, we additionally choose $A>B>0$ such that 
		\begin{align*}
		\kappa := \frac{2B}{\pi (A-B)(2\gamma -1)} <1
		\end{align*}
		holds. Then, we have
		\begin{align*}
			\Delta(\mu, \nu) \leq & \, C_{\gamma, \kappa} \bigg( \int_{-A}^2 \vert G_\mu(u+i) - G_\nu(u+i) \vert du +  \sup_{x \in [-2+\nicefrac{\varepsilon}{2}, 2-\nicefrac{\varepsilon}{2}]} \int_{v}^1 \vert G_\mu(x+iy) - G_\nu(x+iy) \vert dy\\ & \qquad  + 
			\frac{1}{v} \sup_{x \in \mathbb{R}} \int_{\vert y \vert < 2va} \vert F_\nu(x) - F_\nu(x+y) \vert dy +  \pi  \int_{\vert x \vert > B}  \vert F_\mu (x) - F_\nu(x) \vert dx \\ & \qquad  + \gamma \pi\max\{ F_\nu(-2+\varepsilon), 1-F_\nu(2-\varepsilon)\} \bigg),
		\end{align*}
		where $C_{\gamma, \kappa}>0$ is given by $C_{\gamma, \kappa} := ((2\gamma -1)\pi(1-\kappa))^{-1}.$
		
		\begin{proof} 
			Choose  $A > B >0$ as in the premise. Then, with the help of integration by parts and the Fubini-Tonelli theorem, we obtain
			\begin{align*}
				&\int_{-\infty}^{-A} \left \vert   G_\mu(u+i) - G_\nu(u+i) \right \vert du
				\\ &\qquad \leq  \int_{-\infty}^{-A}   \left \vert \int_{-B}^B \frac{F_\mu(x) - F_\nu(x)}{(x-(u+i))^2}dx \right \vert du  +  \int_{-\infty}^{-A}   \left \vert \int_{\vert x \vert > B}\frac{F_\mu(x) - F_\nu(x)}{(x-(u+i))^2}dx \right \vert du   	\\
				& \qquad \leq 
				2B\Delta(\mu, \nu) \int_{-\infty}^{-A} \frac{1}{(u+B)^2} du + \int_{\vert x \vert > B}  \vert F_\mu (x) - F_\nu(x) \vert \left( \int_{-\infty}^{-A} \frac{1}{(x-u)^2+1} du \right) dx\\
				&\qquad  \leq \frac{2B\Delta(\mu, \nu)}{A-B} + \pi  \int_{\vert x \vert > B}  \vert F_\mu (x) - F_\nu(x) \vert dx.
			\end{align*}
			%Fubini-Tonelli is valid, since the function $(x,u) \mapsto \frac{F_\mu(x) - F_\nu(x)}{(x-(u+i))^2} is measurable$
			Combining this calculation with Theorem \ref{Bai Goetze Version more general} and recalling the notation $I_\varepsilon' :=  [-2+\nicefrac{\varepsilon}{2}, 2-\nicefrac{\varepsilon}{2}]$, it follows 
			\begin{align*}
				(2\gamma -1) \pi	\Delta(\mu, \nu) \leq &  \int_{-A}^{2}\vert G_\mu(u+i) - G_\nu(u+i) \vert du +  \frac{2B\Delta(\mu, \nu)}{A-B} + \pi  \int_{\vert x \vert > B}  \vert F_\mu (x) - F_\nu(x) \vert dx \\ &   \,\,\, + \sup_{x \in I_\varepsilon'} \int_{v}^{1}  \vert G_\mu(x+iy) - G_\nu(x+iy) \vert dy  +
				\frac{1}{v} \sup_{x \in \mathbb{R}} \int_{\vert y \vert < 2va} \vert F_\nu(x) - F_\nu(x+y) \vert dy  \\ &  \,\,\, + \gamma \pi\max\{ F_\nu(-2+\varepsilon), 1-F_\nu(2-\varepsilon)\},
			\end{align*}	 
			which implies the claim.
		\end{proof}
	\end{kor}
	We end this section with a remark concerning the integrability condition \eqref{Bai Bed Integral endlich} and the special case that one of the measures appearing in \cref{Bai Goetze Version more general} or \cref{Bai Goetze more general endl. Intgrenzen} is given by Wigner's semicircle law.
	
	\begin{rem} \label{Bai Bem Wigner} 
		Assume that we are in one of the settings given in the last two results.
		\begin{enumerate}[(i)]
			\item As can be seen easily by the layer cake representation, the condition \eqref{Bai Bed Integral endlich} is satisfied if the measures $\mu$ and $\nu$ both have finite second moment.
			\item If we choose $\nu = \omega$, then some of the terms in the previously derived upper bounds on the Kolmogorov distance can be calculated easily. More precisely, we have
			\begin{align*}
			\qquad \,\,\,	\frac{1}{v} \sup_{x \in \mathbb{R}} \int_{\vert y \vert < 2va} \vert F_\omega(x) - F_\omega(x+y) \vert dy \leq \frac{4a^2v}{\pi} \,\,\,\,\, \text{and} \,\,\,\,\,
				\max\{ F_\omega(-2+\varepsilon), 1-F_\omega(2-\varepsilon)\} \leq \frac{\varepsilon^{\nicefrac{3}{2}}}{\pi},
			\end{align*}
			where $F_\omega$ denotes the distribution function of $\omega.$ The first inequality can be proven by using the fact that the density of Wigner's semicircle law is bounded by $\pi^{-1}$, whereas the second one follows easily by symmetry.
		\end{enumerate}
	\end{rem}

	\section{Bounded self-normalized sums} \label{SNS section bounded}
	The main goal of this section is to analyze self-normalized sums of free self-adjoint bounded random variables, i.e.\@ of random variables contained in some $C^*$-algebra. In \cref{SNS section bounded proofs}, we provide the proofs of \cref{bounded_main_non_id} and
	\cref{bounded_main_super,bounded_main_id}. Later, in \cref{SNS section bounded example}, we show how some of the ideas presented in those proofs can be used to establish a convergence result for self-normalized sums of independent GUE matrices.
	
	\subsection{Proofs of Theorem  \texorpdfstring{\ref{bounded_main_non_id}}{1.1} and Corollaries \texorpdfstring{\ref{bounded_main_super}}{1.2} and \texorpdfstring{\ref{bounded_main_id}}{1.3}} \label{SNS section bounded proofs}
	This section is structured as follows: We start by proving the results in \cref{bounded_main_non_id}, from which the statements of \cref{bounded_main_super,bounded_main_id} follow without much effort. Lastly, we comment on an alternative proof of some of the claims in \cref{bounded_main_non_id}.
	
	In the proof of all three results, we work in the following underlying setting: Let $(\mathcal{A}, \varphi)$ be a $C^*$-probability space with unit $1 = 1_{\mathcal{A}} \in \mathcal{A}$, norm $\Vert \cdot \Vert = \Vert \cdot \Vert_{\mathcal{A}}: \mathcal{A} \rightarrow [0, \infty)$, and faithful tracial  $\varphi$. Fix a sequence  $(X_i)_{i \in \mathbb{N}} \subset \mathcal{A}$ of free self-adjoint random variables with $\varphi(X_i) = 0$ and $\varphi(X_i^2) = \sigma_i^2$ for all $i \in \mathbb{N}$. Let $B_n^2, L_{S, 3n}$, and $L_{S, 4n}$ be defined as in \cref{bounded_main_non_id} and assume that $B_n^2 >0$ holds true. In contrast to the definitions given in \cref{bounded_main_non_id}, we set
	\begin{align} \label{bounded - normalized versions of S_n, V_n^2}
		S_n = \frac{1}{B_n} \sum_{i=1}^n X_i, \qquad V_n^2 = \frac{1}{B_n^2} \sum_{i=1}^n X_i^2.
	\end{align}
	Note that this change does not have any effect on the self-normalized sum $U_n = V_n^{-\nicefrac{1}{2}} S_n V_n^{-\nicefrac{1}{2}}$ whenever it exists in $\mathcal{A}$. In the case of existence, we let $\mu_n$ denote the analytic distribution of $U_n$, whereas $F_n$ and $G_n$ denote the distribution function and the Cauchy transform of $\mu_n$. Moreover, the Cauchy transform of the analytic distribution $\mu_{S_n}$ of $S_n$ will be referred to as $G_{S_n}.$ \\ 
	
In order to prove \cref{bounded_main_non_id}, we proceed in three steps: First, we show that $U_n$ is well-defined in $\mathcal{A}$ under certain conditions. For better comparability with the unbounded case, this will be outsourced and done separately in \cref{bounded_Vn is invertible}. The key ingredient in the proof of this lemma is the fact that $V_n^2 $ is close to $1$ in norm.  
	
	Second, in order to establish the claimed rate of convergence, we apply a version of Bai's inequality. Then, it remains to bound the difference between the Cauchy transforms $G_n$ and $G_\omega$. Together with the above-mentioned fact concerning $V_n^2$, we obtain that $G_n$ is close to $G_{S_n}$, whereas the free Berry-Esseen theorem from \cref{Berry Esseen CLT bounded} helps to control the difference between $G_{S_n}$ and $G_{\omega}$. Note that this procedure strongly resembles the intuitive approach used for self-normalized sums in classical probability theory: We consider the sums $S_n$ and $V_n^2$ separately and make use of the free CLT -- in form of the corresponding Berry-Esseen theorem -- and a statement similar to the free LLN -- in form of the closeness of $V_n^2$ to $1$. As already indicated in \cref{SNS sec: Introduction free case}, the machinery of Cauchy transforms serves as a substitute for Slutsky's theorem and thus helps to combine the results concerning $S_n$ and $V_n^2$. 
	
	Third, the convergence claims at the end of \cref{bounded_main_non_id} follow from certain calculations made in the second step without much effort. \\
	
	We begin by proving that the self-normalized sum $U_n$ is well-defined in $\mathcal{A}$. For this purpose, it suffices to ensure that $\smash{V_n^{2}}$ is invertible in $\mathcal{A}$. This can be done with the help of the next two lemmata. The first one (following from \cite[Lemma 3.1.5]{Kadison1997}) is rather standard, but we include it for completeness. %Let us note that the corresponding result is solely based on the fact that $\mathcal{A}$ is a Banach algebra, i.e. the functional $\varphi$ does not play any role in the lemma.	%\cite[Lemma 1.1.10]{Lin2001}. 
	\begin{lem} \label{Lemma Lin}
		Let $a \in \mathcal{A}$ with $\Vert  1-a\Vert <1$. Then, $a$ is invertible in $\mathcal{A}$ and we have
		\begin{align*} a^{-1} = \sum_{k=0}^\infty (1-a)^k , \qquad 	\left \Vert a^{-1} \right \Vert \leq \frac{1}{1 - \Vert 1 -a \Vert}, \qquad \left \Vert 1 - a^{-1}  \right \Vert \leq \frac{\Vert 1-a \Vert}{1- \Vert 1- a \Vert}.
		\end{align*}
	\end{lem}
	
	The second lemma provides an upper bound on the norm of sums of free self-adjoint random variables in $\mathcal{A}$ proven by Voiculescu in \cite[Lemma 3.2]{Voiculescu1986}. In contrast to the formulation given below, Voiculescu's version of the lemma is written in terms of free additive convolutions. Let us remark that this change of perspective is possible since $\varphi$ is faithful. Throughout the rest of this work, we abbreviate $[k] := \{ 1, \dots, k\}$ for $k \in \mathbb{N}$. 
	\begin{lem} \label{bounded lemma voiculescu support}
		Let $a_1, a_2, \dots, a_k$  be free self-adjoint random variables in $\mathcal{A}$ with $\varphi(a_i) = 0$ for all $i \in [k]$. Then, we have
		\begin{align*}
			\left \Vert a_1 + \dots + a_k \right \Vert \leq \max_{i \in [k]} \Vert a_i \Vert + 2 \left( \sum_{i=1}^k \varphi \big(a_i^2 \big)\right)^{\nicefrac{1}{2}}.
		\end{align*}
		\end{lem}

	Using the previous lemmata, we are able to derive a few useful results concerning the random variable $V_n^2$  -- namely its invertibility and closeness to $1$  -- under the assumption that $L_{S,4n}$ is sufficiently small. 
	\begin{lem}\label{bounded_Vn is invertible} Assume that $L_{S, 4n} < \nicefrac{1}{16}$ holds. Then, we have:
		\begin{enumerate}[(i)]
			\item $\left \Vert V_n^2 -1 \right \Vert < 4L_{S, 4n}^{\nicefrac{1}{2}}$,
			\item  $V_n^2$, $V_n$, and $V_n^{\nicefrac{1}{2}}$ are invertible in $\mathcal{A}$.
		\end{enumerate}
		\begin{proof}
			
			By \cref{bounded lemma voiculescu support} and due to the assumption $L_{S, 4n} < \nicefrac{1}{16}$ , we get
			\begin{align*}
				\left 	\Vert V_n^2 -1 \right \Vert = \left \Vert \sum_{i=1}^{n} \frac{X_i^2 - \sigma_i^2}{B_n^2} \right \Vert \leq  \frac{ 2\max_{i \in [n]}  \Vert X_i \Vert^2}{B_n^2} + 2 \left(  \sum_{i=1}^{n} \varphi\left( \frac{(X_i^2 - \sigma_i^2)^2}{B_n^4} \right) \right)^{\nicefrac{1}{2}} < 4L_{S, 4n}^{\nicefrac{1}{2}} <1.
			\end{align*}
			Combining \cref{Lemma Lin} with the inequality above, we obtain that $V_n^2$ is invertible in $\mathcal{A}$. Next, observe that $0 \leq 1 \leq V_n +1$ holds, implying that $V_n +1$ is invertible in $\mathcal{A}$. Its inverse satisfies $\smash{0 \leq (V_n+1)^{-1} \leq 1}$, which leads to $\smash{\Vert (V_n+1)^{-1} \Vert \leq 1}$. Since we can write  $V_n - 1 = (V_n+1)^{-1}(V_n^2 - 1)$, it follows
			\begin{align*}
				\Vert V_n-1 \Vert \leq \left\Vert (V_n+1)^{-1} \right \Vert  \left\Vert V_n^2 -1 \right \Vert \leq \left \Vert V_n^2 -1 \right\Vert <1.
			\end{align*}
			Hence, by \cref{Lemma Lin}, we deduce that $V_n$ is invertible in $\mathcal{A}$. The invertibility of  $\smash{V_n^{\nicefrac{1}{2}}}$ can be shown analogously.
		\end{proof}
	\end{lem}
	
	Now, let us complete the proof of \cref{bounded_main_non_id}.
	\begin{proof}[Proof of Theorem \ref{bounded_main_non_id}]  
		In the setting fixed at the beginning of this section, assume that	$L_{S,4n} < \nicefrac{1}{16}$ holds. 
		According to \cref{bounded_Vn is invertible}, $U_n$ is well-defined in $\mathcal{A}$, which proves the first claim in \cref{bounded_main_non_id}. 
		
		We continue by proving the second claim of the theorem dealing with the rate of convergence of $\mu_n$ to Wigner's semicircle law $\omega$ under the conditions $L_{S, 4n} < \nicefrac{1}{16}$ and $L_{S, 3n}< \nicefrac{1}{2e}$. Let us temporarily assume the stronger condition $\smash{L_{S, 4n} < \nicefrac{1}{(160e)^2}}$. Then, we can apply Bai's inequality as given in \cref{Bai Goetze more general endl. Intgrenzen} with the parameters
		\begin{align*} 
			A = 8, \qquad B = 3, \qquad  a=2, \qquad \gamma > 0.7, \qquad v = \max\left\{ L_{S, 3n}, 80L_{S, 4n}^{\nicefrac{1}{2}}\right\}, \qquad \varepsilon =6v.
		\end{align*} 
		Note that the integrability of $\vert F_n - F_\omega \vert$ as requested in \eqref{Bai Bed Integral endlich} is guaranteed due to the fact that $U_n$ is an element in $\mathcal{A}$ and thus $\mu_n$ has finite second moment; compare to \cref{Bai Bem Wigner}. 
		
		Let us proceed by showing that $\mu_n$ has compact support in $[-3,3]$. This will allow us to simplify Bai's inequality; see \eqref{bounded - Bai intermediate} below. In view of the calculations made in the proof of \cref{bounded_Vn is invertible}, we have 
		\begin{align*}
			\Vert V_n - 1 \Vert \leq \big \Vert V_n^2 -1 \big \Vert \leq 4L_{S, 4n}^{\nicefrac{1}{2}} < \frac{1}{40}
		\end{align*}
		leading to $\Vert V_n^{-1} \Vert < \nicefrac{40}{39}$ with the help of \cref{Lemma Lin}.
		By \cref{bounded lemma voiculescu support}, we obtain
		\begin{align}  \label{bounded_support ungleichung_Sn}
			\left \Vert S_n  \right \Vert \leq  2 +\frac{ \max_{i \in [n]} \Vert X_i \Vert }{B_n} \leq 2+ L_{S,4n}^{\nicefrac{1}{4}} < \frac{21}{10}.
		\end{align}
		It follows
		\begin{align}  \label{bounded_support ungleichung}
			\Vert U_n \Vert = \left \Vert V_n^{-\nicefrac{1}{2}} S_n V_n^{-\nicefrac{1}{2}}  \right \Vert \leq \left \Vert V_n^{-\nicefrac{1}{2}} \right \Vert^2  \left\Vert S_n \right\Vert =  \left \Vert V_n^{-1} \right \Vert  \left\Vert S_n \right\Vert \leq \frac{40}{39} \left \Vert S_n \right\Vert <3,
		\end{align}
		which proves $\supp \mu_n \subset [-3, 3]$ as claimed above. Combining this with $\supp \omega = [-2,2]$ and \cref{Bai Bem Wigner}, Bai's inequality from \cref{Bai Goetze more general endl. Intgrenzen} reduces to
		\begin{align} \label{bounded - Bai intermediate}
			\begin{split}
				C_{\gamma, \kappa}^{-1}\Delta\left(\mu_n, \omega \right) & \leq \frac{16v}{\pi} + \gamma \varepsilon^{\nicefrac{3}{2}} + \int_{-8}^2 \vert G_n(u+i)-G_\omega(u+i)  \vert du \\ & \qquad + \sup_{x \in [-2 + \nicefrac{\varepsilon}{2}, 2 - \nicefrac{\varepsilon}{2}]}  \int_{v}^1 \vert G_n(x+iy) - G_\omega(x+iy) \vert dy
			\end{split}
		\end{align}
		with $C_{\gamma, \kappa}$ being the constant from \cref{Bai Goetze more general endl. Intgrenzen}.
		
		It suffices to bound the difference between the Cauchy transforms $G_n$ and $G_\omega$ evaluated at certain $z \in A_v := \{ z \in \mathbb{C}^+: \vert \Re z \vert \leq 8, 1 \geq  \Im z  \geq v\}.$
		Before we continue with this, let us do some preparatory work, mainly consisting of the analysis of the random variables $(zV_n - S_n)^{-1}$ and $(z-S_n)^{-1}$ for $z \in A_v.$  %Observe that $z-U_n$ and $z-S_n$ both are invertible in $\mathcal{A}$ for any $z \in \mathbb{C}^+$. 
		Due to
		\begin{align*}
			zV_n - S_n = V_n^{\nicefrac{1}{2}} (z - U_n) V_n^{\nicefrac{1}{2}}
		\end{align*}
		and since $V_n^{\nicefrac{1}{2}}$ is invertible, we obtain that $zV_n - S_n$ is invertible in $\mathcal{A}$ for any $z \in \mathbb{C}^+$.
		For $z \in A_v$, define $W_z:=z(z-S_n)^{-1}(V_n-1) \in \mathcal{A}$ and $U_z:= 1 + W_z \in \mathcal{A}$. Then, because of $v \geq 20 \Vert V_n - 1\Vert$, we have 
		\begin{align*}
			\Vert W_z \Vert \leq \frac{\vert z \vert}{\Im z} \Vert V_n-1 \Vert \leq \frac{9}{ \Im z }\Vert V_n-1 \Vert \leq  \frac{9}{v}\Vert V_n-1 \Vert < \frac{1}{2}
		\end{align*}
		for any $z \in A_v.$
		Thus, by \cref{Lemma Lin}, $U_z$ is invertible in $\mathcal{A}$ with
		\begin{align*}
			\left \Vert U_z^{-1}  \right \Vert \leq \frac{1}{1- \Vert W_z \Vert} < 2
		\end{align*}
		for $z \in A_v$.
		Consequently, we can write
		\begin{align*}
			(zV_n-S_n)^{-1} = \left( (z-S_n) \left( 1 +  z(z-S_n)^{-1}(V_n-1) \right) \right)^{-1} = U_z^{-1} (z-S_n) ^{-1}, \qquad z \in A_v.
		\end{align*}
		Using the inequality $0 \leq y^*x^*xy \leq \Vert x \Vert^2 \vert y \vert^2$ holding for any $x,y \in \mathcal{A}$, we get
		\begin{align*}
			\begin{split}
				\varphi \left(\left \vert \left(zV_n - S_n  \right)^{-1}\right\vert^2\right)  
				%= \varphi\left( \left(U_z^{-1}(z-S_n)^{-1}\right)^* \left(U_z^{-1}(z-S_n)^{-1}\right) \right)  
				= \varphi\left((\bar{z}-S_n)^{-1}\big(U_z^{-1}\big)^* U_z^{-1}(z-S_n)^{-1} \right)	<  4\varphi\left(\left\vert (z-S_n)^{-1} \right\vert^2 \right)
			\end{split}
		\end{align*}
		for all $z \in A_v.$
		The resolvent identity implies
		\begin{align*}
			\begin{split}
				\varphi\left(\left\vert (z-S_n)^{-1} \right\vert^2 \right)
				&=\varphi\left( (\bar{z}-S_n)^{-1}(z - S_n)^{-1} \right)
				= \frac{\varphi\left((\bar{z}-S_n)^{-1} \right) - \varphi\left( (z - S_n)^{-1} \right)}{z-\bar{z}} \\ &= - \frac{\Im \left(\varphi\left( (z-S_n)^{-1}\right) \right)}{\Im z} =  \frac{ \vert \Im \left( G_{S_n}(z) \right) \vert}{\Im z} 
			\end{split}
		\end{align*}
		for all $z \in \mathbb{C}^+.$ By integration by parts and the inequality $\vert G_\omega \vert \leq 1$ holding in $\mathbb{C}^+$ (see \cite[Lemma 8]{Kargin2007a}), we obtain
		\begin{align*}
			\vert G_{S_n}(z) \vert \leq 1+ \vert G_{S_n}(z) - G_\omega(z) \vert \leq  1 + \frac{\pi \Delta(\mu_{S_n}, \omega)}{\Im z}, \qquad z \in \mathbb{C}^+. 
		\end{align*}
		\cref{Berry Esseen CLT bounded} yields 
		\begin{align*}
			\vert G_{S_n}(z) \vert  \leq 1 + \frac{\pi C_0 L_{S,3n}}{\Im z} \leq 1 + \pi C_0
		\end{align*}
		for all $z \in \mathbb{C}^+$ with $\Im z \geq L_{S, 3n}$, where $C_0>0$ is taken from \cref{Berry Esseen CLT bounded}. It follows
		\begin{align*}
			\varphi\left( \left \vert (z-S_n)^{-1}\right \vert^2 \right) \leq \frac{\vert G_{S_n}(z) \vert}{\Im z} \leq \frac{1 + \pi C_0}{\Im z}, \qquad \Im z \geq L_{S, 3n}.
		\end{align*}
		
		Now, let us analyze the difference of $G_n$ and $G_{S_n}$ evaluated at some $z \in A_v$. Since $\varphi$ is tracial, we get 
		\begin{align} \label{bounded_expansion G_n}
			\begin{split}
				G_n(z)& = \varphi\left(\left( z - U_n \right)^{-1} \right) = \varphi \left(V_n^{\nicefrac{1}{2}} (zV_n -S_n)^{-1} V_n^{\nicefrac{1}{2}} \right) \\ & =\varphi\left(V_n (zV_n -S_n)^{-1} \right) \! = \! \varphi\left( \left(V_n - 1\right) (zV_n -S_n)^{-1} \right) + \varphi \left((zV_n -S_n)^{-1} \right)
			\end{split}
		\end{align}
	for all $z \in \mathbb{C}^+$, which leads to
		\begin{align*} 
			G_n(z) - G_{S_n}(z)  = \varphi\left( \left(V_n - 1\right)U_z^{-1} (z-S_n)^{-1} \right) + \varphi \left( (zV_n - S_n)^{-1} - (z -S_n)^{-1}\right)
		\end{align*}
		for all $z \in A_v.$
		Combining our preparatory work with the Cauchy-Schwarz inequality for $\varphi$ and using $\vert \varphi(x) \vert \leq \Vert x \Vert$ holding for all $x \in \mathcal{A}$, the first summand on the right-hand side in the equation above admits the estimate 
		\begin{align*}
			\left \vert \varphi\left( (V_n -1)U_z^{-1}(z - S_n)^{-1}\right) \right \vert & \leq  \varphi\left( (V_n -1)U_z^{-1} (U_z^{-1})^*(V_n -1) \right)^{\nicefrac{1}{2}} \varphi\left( \left \vert (z-  S_n)^{-1} \right \vert^2\right) ^{\nicefrac{1}{2}}  \\ &\leq
			\left \Vert U_z^{-1} \right \Vert  \left\Vert V_n-1 \right\Vert  \varphi\left( \left \vert (z-  S_n)^{-1} \right \vert^2\right) ^{\nicefrac{1}{2}} \leq \frac{2 \sqrt{1+\pi C_0} }{\sqrt{\Im z}}\left\Vert V_n-1 \right\Vert
		\end{align*}
		for any $z \in A_v$. Similarly, we deduce 
		\begin{align*}
			\left \vert \varphi \left((zV_n -S_n)^{-1} - (z -S_n)^{-1}\right) \right \vert  & =  \left \vert \varphi \left( (zV_n - S_n)^{-1}z(1-V_n)(z- S_n)^{-1}\right) \right \vert \\  & \leq  \vert z \vert \left \Vert V_n  - 1 \right\Vert \varphi \left(\left \vert \left(zV_n - S_n  \right)^{-1}\right\vert^2\right)^{\nicefrac{1}{2}} \varphi\left( \left \vert (z- S_n)^{-1} \right \vert^2 \right)^{\nicefrac{1}{2}} \\
			& \leq  2\vert z \vert \left \Vert V_n -1 \right \Vert  \varphi\left(\left\vert (z-S_n)^{-1} \right\vert^2 \right)  \leq  \frac{2(1+ \pi C_0)\vert z \vert}{\Im z} \left \Vert V_n -1 \right \Vert  
		\end{align*}
		for all $z$ as above. Hence, we arrive at
		\begin{align*} 
			\left \vert G_n(z) - G_{S_n}(z) \right \vert \leq \left(\frac{\sqrt{1 + \pi C_0}}{\sqrt{\Im z}} + \frac{(1 +  \pi C_0)\vert z \vert}{\Im z} \right)8L_{S,4n}^{\nicefrac{1}{2}}  , \qquad z \in A_v.
		\end{align*}
		Integration yields
		\begin{align*}
			\int_{-8}^2  \vert G_n(u+i)  - G_{S_n}(u+i) \vert  du &\leq   \left( 10\sqrt{1 + \pi C_0}+  (1+\pi C_0) \int_{-8}^{2}\sqrt{1+u^2}  du \right) 8L_{S,4n}^{\nicefrac{1}{2}} \\ & \leq   \left( 80\sqrt{1+ \pi C_0} + 293(1+\pi C_0)\right) L_{S,4n}^{\nicefrac{1}{2}}
		\end{align*}
		and
		\begin{align*}
			\int_{v}^{1}   \vert G_n(x+iy) - G_{S_n}(x+iy) \vert dy& \leq \left( \int_{v}^{1} \frac{\sqrt{1 + \pi C_0}}{\sqrt{y}} dy  + (1+\pi C_0)\int_v^1\frac{\vert x + iy \vert}{y} dy \right)8L_{S,4n}^{\nicefrac{1}{2}}   \\
			& \leq  \left(16 \sqrt{1 + \pi C_0}+ 8(1+\pi C_0) + 16(1+\pi C_0)\vert \! \log v\vert\right) L_{S,4n}^{\nicefrac{1}{2}}
		\end{align*}
		for any $x \in [-2+\nicefrac{\varepsilon}{2}, 2-\nicefrac{\varepsilon}{2}]$.

		It remains to handle the integral contributions from the difference of $G_{S_n}$ and $G_\omega$. 
		We have
		\begin{align*}
			\int_{v}^{1}   \vert G_{S_n}(x+iy) - G_\omega(x+iy) \vert dy \leq \int_{v}^{1}  \frac{\pi C_0 L_{S,3n} }{y}dy \leq \pi C_0 L_{S, 3n} \vert \!\log v \vert
		\end{align*}
		for $x$ as above and 
		\begin{align*}
			\int_{-8}^2  \vert  G_{S_n}(u+i) - G_\omega(u+i) \vert du \leq 10\pi C_0L_{S,3n}.
		\end{align*}
		Using \eqref{bounded - Bai intermediate}, we conclude
		\begin{align} \label{bounded non id intermediate for remark}
			C_{\gamma, \kappa}^{-1}\Delta\left(\mu_n, \omega \right)  \leq \frac{16v}{\pi} + \gamma \varepsilon^{\nicefrac{3}{2}} + \pi C_0 L_{S, 3n} \left( 10 + \vert \!\log v \vert\right)
			+ L_{S, 4n}^{\nicefrac{1}{2}} \left(C_1+ C_2\vert \!\log v \vert \right)
		\end{align}
		for suitably chosen constants $C_1, C_2>0$.
		
		Observing that $v < \nicefrac{1}{e}$ holds, we obtain
		\begin{align*}
			\max \left\{ L_{S, 3n}, L_{S, 4n}^{\nicefrac{1}{2}} \right\}  & \leq 	\max \left\{ L_{S, 3n}, L_{S, 4n}^{\nicefrac{1}{2}} \right\} \vert \! \log v \vert \\ & \leq 	\max\left\{ L_{S, 3n},L_{S, 4n}^{\nicefrac{1}{2}}\right\} \left \vert \log\left( \max \left\{L_{S, 3n}, L_{S, 4n}^{\nicefrac{1}{2}} \right\} \right) \right\vert \\ & \leq \max\left\{\vert \!\log L_{S, 3n} \vert L_{S, 3n}, \vert \!\log L_{S, 4n}\vert L_{S, 4n}^{\nicefrac{1}{2}}\right\}.
		\end{align*}
		Together with \eqref{bounded non id intermediate for remark}, this implies
		\begin{align*}
			\Delta\left(\mu_n, \omega \right)  &\leq C_3\max\left\{\vert \!\log L_{S, 3n} \vert L_{S, 3n}, \vert \!\log L_{S, 4n}\vert L_{S, 4n}^{\nicefrac{1}{2}}\right\}
		\end{align*}
		for some constant $C_3>0.$ 
		
		Recall that the inequality above only holds under the stronger condition $\smash{L_{S, 4n} < \nicefrac{1}{(160e)^2}}$. Let us now consider the remaining case, i.e.\@ we are given  $\smash{\nicefrac{1}{16} > L_{S, 4n} \geq \nicefrac{1}{(160e)^2}}$ and $L_{S, 3n} < \nicefrac{1}{2e}$. Due to $\smash{L_{S, 4n} \leq L_{S, 3n}^{\nicefrac{4}{3}} \leq L_{S, 3n}}$, we obtain
		$\smash{L_{S, 3n} \geq \nicefrac{1}{(160e)^2}}$, which leads to
		\begin{align*}
			\max\left\{\vert \!\log L_{S, 3n} \vert L_{S, 3n}, \vert \!\log L_{S, 4n}\vert L_{S, 4n}^{\nicefrac{1}{2}}\right\} \geq  \frac{ \max\{ \vert \! \log L_{S, 3n} \vert, \vert \! \log L_{S, 4n} \vert \}}{(160e)^2} \geq  \frac{ \vert \! \log L_{S, 4n} \vert}{(160e)^2} > 
			\frac{2}{(160e)^2}.
		\end{align*}
		In particular, we arrive at
		\begin{align*}
			\Delta(\mu_n, \omega) \leq 1 \leq \frac{(160e)^2}{2}\max\left\{\vert \!\log L_{S, 3n} \vert L_{S, 3n}, \vert \!\log L_{S, 4n}\vert L_{S, 4n}^{\nicefrac{1}{2}}\right\}.
		\end{align*}	
		Hence, by setting $C := \max\{ C_3, \nicefrac{(160e)^2}{2}\}$, the claim concerning the rate of convergence follows in all cases.
		
		Let us end by verifying the last two claims of \cref{bounded_main_non_id}. Assume that $\lim_{n \rightarrow \infty} L_{S, 4n} = 0 $ holds. Then, the free Lindeberg CLT is applicable and we get $\lim_{n \rightarrow \infty} G_{S_n}(z) = G_\omega(z)$ for all $z \in \mathbb{C}^+$. Combining \eqref{bounded_expansion G_n} with the inequality
		\begin{align*}
			\Vert (zV_n - S_n)^{-1} \Vert \leq \frac{40}{39} \frac{1}{\Im z}, \qquad z \in \mathbb{C}^+,
		\end{align*}
		it follows
		\begin{align*}
			\vert G_n(z) - G_{S_n}(z) \vert \leq \frac{40}{39} \Vert V_n -1 \Vert \left( \frac{1}{\Im z}+  \frac{\vert z \vert}{(\Im z)^2} \right)  \leq \frac{160}{39} \left( \frac{1}{\Im z}+  \frac{\vert z \vert}{(\Im z)^2} \right)  L_{S, 4n}^{\nicefrac{1}{2}}
		\end{align*}
		for all $z \in \mathbb{C}^+$. Thus, we conclude $\lim_{n \rightarrow \infty} G_{n}(z) = G_\omega(z)$ for all $z \in \mathbb{C}^+$ leading to $\mu_n \Rightarrow \omega$ as $n \rightarrow \infty$. Due to $\supp \mu_n \subset [-3,3]$ holding for sufficiently large $n$ and  \cite[Theorem 25.12]{Billingsley2012}, the claimed convergence of moments follows from the just-proven weak convergence. 
	\end{proof}
	
	Before we continue with proving the remaining claims on bounded self-normalized sums, let us briefly comment on the assumptions made in \cref{bounded_main_non_id} as well as on possible variations of our proof.
	\begin{rem} \label{bounded comment on proof of main thm non id}
		\begin{enumerate}[(i)]
			\item In our proof, the assumption $L_{S, 4n}<\nicefrac{1}{16}$ is necessary in order to guarantee that $U_n$ is a well-defined random variable in $\mathcal{A}$, whereas the condition $L_{S,3n} < \nicefrac{1}{2e}$ is made mainly for convenience. In fact, it is possible to obtain the same rate of convergence (up to constants) as in \cref{bounded_main_non_id} under the assumptions $1 > c_0 > L_{S, 3n}$ for some $c_0 \in (0,1)$  and $L_{S, 4n} < \nicefrac{1}{16}$ by replacing the constant $C$ appearing in the last-named theorem by $C' := \max\{C, 2e(\vert \!\log c_0 \vert)^{-1} \}$.
			\item The rate of convergence established in \cref{bounded_main_non_id} depends on $L_{S,3n}$ due to the fact that the speed of convergence in the free CLT is of order $L_{S,3n}$; compare to \cref{Berry Esseen CLT bounded}. Since \cref{Berry Esseen CLT bounded} additionally provides an alternative rate in the free CLT of order $\smash{L_{3n}^{\nicefrac{1}{2}}}$, one can also bound $\Delta(\mu_n, \omega)$ in terms of the quantities  $\smash{L_{S, 4n}^{\nicefrac{1}{2}}}$ and  $\smash{L_{3n}^{\nicefrac{1}{2}}}$ -- up to some logarithmic factors and under suitable conditions on both fractions. 
		\end{enumerate}	
	\end{rem}
	
	Let us now study the support of $\mu_n$. According to the calculations in \eqref{bounded_support ungleichung_Sn} and \eqref{bounded_support ungleichung}, we already know that 
	\begin{align*}
		\supp \mu_n \subset \left[ -\frac{40}{39}\left(2+ \frac{\max_{i \in [n]} \Vert X_i \Vert}{B_n}\right), \frac{40}{39}\left(2+ \frac{\max_{i \in [n]} \Vert X_i \Vert}{B_n}\right)\right]
	\end{align*}
	holds, whenever $L_{S, 4n}$ is sufficiently small. A modification of those calculations leads to the result in \cref{bounded_main_super}.
	\begin{proof}[Proof of \cref{bounded_main_super}] In the setting introduced at the beginning of this section, assume that $L_{S, 4n} < \nicefrac{1}{64}$ holds. In view of \cref{Lemma Lin} and the calculations made in the proof of \cref{bounded_Vn is invertible}, we have 
		\begin{align*}
			\left \Vert V_n^{\nicefrac{1}{2}} - 1 \right \Vert \leq \Vert V_n - 1 \Vert \leq 4L_{S, 4n}^{\nicefrac{1}{2}} < \frac{1}{2}, \qquad \left \Vert V_n^{-\nicefrac{1}{2}} -1 \right \Vert < 2\left  \Vert V_n^{\nicefrac{1}{2}} -1 \right \Vert \leq 8L_{S, 4n}^{\nicefrac{1}{2}}.
		\end{align*}
		Together with \eqref{bounded_support ungleichung_Sn}, it follows
		\begin{align*}
			\Vert U_n \Vert &\leq \left\Vert \left(V_n^{-\nicefrac{1}{2}} -1 \right)S_n\left(V_n^{-\nicefrac{1}{2}} -1\right) \right\Vert + \left\Vert \left(V_n^{-\nicefrac{1}{2}} -1 \right)S_n \right\Vert + \left \Vert S_nV_n^{-\nicefrac{1}{2}} \right\Vert \\ & \leq \left \Vert V_n^{-\nicefrac{1}{2}} -1 \right\Vert^2 \Vert S_n \Vert + 2 \left \Vert V_n^{-\nicefrac{1}{2}} - 1 \right \Vert \Vert S_n \Vert + \Vert S_n \Vert \\ & < \left(2+L_{S,4n}^{\nicefrac{1}{4}}\right)\left(64L_{S, 4n} + 16L_{S, 4n}^{\nicefrac{1}{2}}\right) + 2 + \frac{\max_{i \in [n]} \Vert X_i \Vert}{B_n} \\ &\leq 2 + \frac{\max_{i \in [n]} \Vert X_i \Vert}{B_n}  
			+57L_{S, 4n}^{\nicefrac{1}{2}}.
		\end{align*}
	\end{proof}
	
	Lastly, we prove \cref{bounded_main_id} dealing with the special case of identical distributions. 
	\begin{proof}[Proof of \cref{bounded_main_id}]
		Assume that the sequence $(X_i)_{i \in\mathbb{N}}$ fixed at the beginning of this section is a sequence of free identically distributed self-adjoint random variables with $\varphi(X_1) = 0$ and $\varphi(X_1^2) = 1$. 
		Then, it follows $\smash{B_n^2 = n >0,  L_{S, 3n} = \Vert X_1 \Vert^3n^{-\nicefrac{1}{2}}}$, and $L_{S, 4n} = \Vert X_1 \Vert ^4n^{-1}.$ From \cref{bounded_main_non_id}, we know that the self-normalized sum $U_n$ is well-defined in $\mathcal{A}$ for $n > 16\Vert X_1 \Vert ^4$ and that its analytic distribution $\mu_n$ satisfies 
		\begin{align*}
			\Delta(\mu_n, \omega) \leq 2C_0\Vert X_1 \Vert^3 \frac{\log n}{\sqrt{n}}
		\end{align*}
		for $n > \max\{16\Vert X_1 \Vert^4, 4e^2\Vert X_1 \Vert^6\}$ with $C_0>0$ being the constant in \cref{bounded_main_non_id}. In the case that $16\Vert X_1 \Vert^4 < n \leq 4e^2\Vert X_1 \Vert^6$ holds, we have $n^{-\nicefrac{1}{2}}\log n > (2e\Vert X_1 \Vert^3)^{-1}$, 
		which leads to 
		\begin{align*}
			\Delta(\mu_n, \omega) \leq 1 \leq 2e\Vert X_1 \Vert^3 \frac{\log n}{\sqrt{n}}. 
		\end{align*}
		Letting $C := 2\max\{ C_0, e\}$, we arrive at $
		\Delta(\mu_n, \omega) \leq  C\Vert X_1 \Vert^3(\log n) n^{-\nicefrac{1}{2}}$ for $n > 16\Vert X_1 \Vert^4$ as claimed.  The weak convergence and the moment convergence of $\mu_n$ to $\omega$ as $n \rightarrow \infty$ both follow from \cref{bounded_main_non_id}, whereas the statement concerning the support of $\mu_n$ is immediate from
		\cref{bounded_main_super}. 
	\end{proof}

	In \cref{SNS sec: Introduction free case}, we indicated that the use of Cauchy transforms is not the only option to substitute Slutsky's theorem in the intuitive approach to self-normalized sums. As we will see in the following remark, one can use a convergence result for \textit{rational expressions} in strongly convergent bounded random variables in order to prove the convergence of moments (and thus the weak convergence) stated in \cref{bounded_main_non_id}. We refer to  \cite{Collins2022, Yin2018} for the definition of non-commutative rational expressions.
	
	\begin{rem} \label{SNS bounded Remark Yin} 
		We work in the setting fixed at the beginning of this section. Assuming $\lim_{n \rightarrow \infty} L_{S,4n} = 0$ and making use of \cite[Theorem 1.2]{Yin2018}, our goal is to show that the self-normalized sum $U_n$ satisfies 
	\begin{align*}
	\lim_{n \rightarrow \infty} \varphi ( U_n^k ) = \varphi (s^k), \qquad k \in \mathbb{N},
	\end{align*}
	where $s \in \mathcal{A}$ is a standard semicircular element.
		
		Note that we find $n_0 \in \mathbb{N}$ in such a way that $L_{S, 4n}< \nicefrac{1}{64}$ holds for all $n \geq n_0.$ Then, by \cref{bounded_Vn is invertible}, the random variable $\smash{V_n^2}$ is invertible in $\mathcal{A}$, implying that $U_n$ is well-defined for those $n$. For any $k \in \mathbb{N}$, define the rational expression $r_k$ by $\smash{r_k(x,y) := (xy^{-1})^k}$ and observe that we have $\smash{\varphi(U_n^k)  = \varphi(r_k(S_n, V_n))}$ for any $n \geq n_0$  (since $\varphi$ is tracial) as well as $\smash{r_k(s,1) = s^k \in \mathcal{A}}$. If we assume that $(S_n, V_n)$ strongly converges to $(s,1)$ as $n \rightarrow \infty$, then we can apply \cite[Theorem 1.2]{Yin2018} to $r_k$ leading to 
		\begin{align*}
			\lim_{n \rightarrow \infty}\varphi(U_n^k)  = \lim_{n \rightarrow \infty}\varphi(r_k(S_n, V_n)) = \varphi(r_k(s,1)) = \varphi(s^k)
		\end{align*}
		for all $k \in\mathbb{N}$ as claimed. 
		
		It remains to verify that $(S_n, V_n)$ strongly converges to $(s,1)$ as $n \rightarrow \infty$. For this, we have to prove 
		\begin{align*} 
			\begin{split} 
				&\lim_{n \rightarrow \infty} \varphi ( P(S_n, V_n)) = \varphi( P(s,1)),  \qquad \lim_{n \rightarrow \infty} \Vert P(S_n, V_n) \Vert = \Vert P(s,1) \Vert
			\end{split}
		\end{align*}
		for any polynomial $P$ in two non-commuting formal variables. Let us begin with the convergence in $\varphi.$ Due to linearity, it suffices to establish convergence of all mixed moments, i.e.
		\begin{align*}
			\lim_{n \rightarrow \infty} \varphi ( S_n^{p_1}V_n^{q_1} S_n^{p_2}V_n^{q_2} \cdots S_n^{p_m} V_n^{q_m}) = \varphi( s^{p_1 + p_2 + \cdots + p_m})
		\end{align*}
		for all choices of $p_1, q_1, p_2, q_2, \dots, p_m, q_m, m \in \mathbb{N}$. Below, we restrict to $m=3$; the general case can be dealt with in complete analogy. Fix $p_1, q_1, p_2, q_2, p_3, q_3 \in \mathbb{N}$. With the help of \cref{bounded lemma voiculescu support,bounded_Vn is invertible}, we get 
		\begin{align*}
			\Vert S_n \Vert \leq 2+ L_{S, 4n}^{\nicefrac{1}{4}} < 3, \qquad \Vert V_n \Vert \leq 1+ \Vert V_n -1 \Vert \leq 1+ 4L_{S,4n}^{\nicefrac{1}{2}} < 2
		\end{align*}
		for $n \geq n_0$. Moreover, note that we have
		\begin{align*}
			\Vert V_n^{q_j} - 1 \Vert = \left\Vert (V_n - 1) \sum_{i=1}^{q_j} V_n^{i-1} \right\Vert \leq \Vert V_n - 1 \Vert \sum_{i=1}^{q_j} \Vert V_n \Vert^{i-1} < 4L_{S, 4n}^{\nicefrac{1}{2}}
			\sum_{i=1}^{q_j} 2^{i-1} \rightarrow 0
		\end{align*}
		as $n \rightarrow \infty$ for any $j \in [3]$. Together with 
		\begin{align*}
			\begin{split}
				S_n^{p_1}V_n^{q_1} S_n^{p_2}V_n^{q_2}S_n^{p_3} V_n^{q_3} &= S_n^{p_1}(V_n^{q_1} - 1)S_n^{p_2} V_n^{q_2} S_n^{p_3} V_n^{q_3} + S_n^{p_1+p_2}V_n^{q_2}S_n^{p_3}V_n^{q_3}  \\ &= S_n^{p_1}(V_n^{q_1} - 1)S_n^{p_2} V_n^{q_2} S_n^{p_3} V_n^{q_3} + S_n^{p_1+p_2}(V_n^{q_2}-1)S_n^{p_3}V_n^{q_3} + S_n^{p_1+p_2+p_3}V_n^{q_3} \\ & =S_n^{p_1}(V_n^{q_1} - 1)S_n^{p_2} V_n^{q_2} S_n^{p_3} V_n^{q_3} + S_n^{p_1+p_2}(V_n^{q_2}-1)S_n^{p_3}V_n^{q_3} \\ &  \qquad + S_n^{p_1+p_2+p_3}(V_n^{q_3}-1) + S_n^{p_1+p_2+p_3},
			\end{split}
		\end{align*}
		it follows
		\begin{align*}
			& \big \vert \varphi ( S_n^{p_1}V_n^{q_1} S_n^{p_2}V_n^{q_2}S_n^{p_3} V_n^{q_3}) - \varphi(S_n^{p_1+p_2+p_3}) \big \vert \\ & \qquad \leq  \Vert S_n \Vert^{p_1 + p_2 + p_3} \left( \Vert V_n^{q_1} -1 \Vert \Vert V_n \Vert^{q_2 + q_3} + \Vert V_n^{q_2}-1 \Vert \Vert V_n \Vert^{q_3} + \Vert V_n^{q_3} -1 \Vert  \right)\\ & \qquad \leq 3^{p_1+p_2+p_3} \left(2^{q_2 + q_3}\Vert V_n^{q_1} -1 \Vert + 2^{q_3}\Vert V_n^{q_2}-1 \Vert + \Vert V_n^{q_3} -1 \Vert  \right) \rightarrow 0
		\end{align*}
		as $n \rightarrow \infty$. By the free Lindeberg CLT, which is applicable due to $\lim_{n \rightarrow \infty} L_{S, 4n} = 0$, we know that $\mu_{S_n} \Rightarrow \omega$ holds as $n \rightarrow \infty$. Together with $\supp \mu_{S_n} \subset [-3,3]$ for all $n \geq n_0$, we arrive at
		\begin{align*}
			\lim_{n \rightarrow \infty} \varphi(S_n^{p_1 + p_2 + p_3}) = \varphi(s^{p_1 + p_2 + p_3}).
		\end{align*}
		  Combining the last observations, we obtain 
		  \begin{align*}
		  	\lim_{n \rightarrow \infty} \varphi(	S_n^{p_1}V_n^{q_1} S_n^{p_2}V_n^{q_2}S_n^{p_3} V_n^{q_3}) = \varphi(s^{p_1 + p_2 + p_3})
		  \end{align*}
		  as claimed.
		
		In order to verify the requested norm convergence, let $P$ be a polynomial in two non-commuting formal variables. Making use of \cite[Proposition 2.1]{Collins2014}, 
		we get
		$\lim_{n \rightarrow \infty} \Vert P(S_n, 1) \Vert = \Vert P(s,1) \Vert$. Hence, the claimed convergence follows if we can prove $\lim_{n \rightarrow \infty} \left  \Vert P(S_n, V_n) -P(S_n, 1) \right \Vert = 0.$
		Here, it suffices to show 
		\begin{align*}
			\lim_{n \rightarrow \infty} \Vert S_n^{p_1}V_n^{q_1} S_n^{p_2}V_n^{q_2} \cdots S_n^{p_m} V_n^{q_m} -S_n^{p_1+p_2 + \cdots + p_m} \Vert = 0 
		\end{align*}
		for all choices of $p_1, q_1, p_2, q_2, \dots, p_m, q_m, m \in \mathbb{N}$, which can be done easily by using the previous calculations. For instance, in the case $m =3$, we have
		\begin{align*}
			&\left \Vert S_n^{p_1}V_n^{q_1} S_n^{p_2}V_n^{q_2}S_n^{p_3} V_n^{q_3} - S_n^{p_1+p_2+p_3}  \right \Vert \\ & \qquad \leq  3^{p_1+p_2+p_3} \left(2^{q_2 + q_3}\Vert V_n^{q_1} -1 \Vert + 2^{q_3}\Vert V_n^{q_2}-1 \Vert + \Vert V_n^{q_3} -1 \Vert  \right) \rightarrow 0
		\end{align*}
		as $n \rightarrow \infty.$
	\end{rem}
		
	%Note that the approach via the resolvents as presented above immediately implies the corresponding rate of convergence, whereas the approach via strong convergence does not. 
	
	\subsection{\texorpdfstring{Example: Self-normalized sums of independent GUE matrices}{Example: Self-normalized sums of independent GUE matrices}} \label{SNS section bounded example}
	The goal of this section is to show how some of the arguments provided in the previous section, especially those belonging to the intuitive approach to self-normalized sums, can be used to derive a self-normalized limit theorem in the context of independent GUE matrices.
	
	Before we formulate the exact statement, let us briefly recall some aspects from random matrix theory.  For $N \in \mathbb{N}$, let $M_N(\mathbb{C})$ denote the $C^*$-algebra of matrices of dimension $N \times N$ over $\mathbb{C}$. For any matrix $X \in M_N(\mathbb{C})$, the normalized trace of $X$ is denoted by  $\tr_N(X)$, whereas $\Vert X\Vert_N$ is the operator norm given by the largest singular value of $X$. A self-adjoint random matrix $X = N^{-\nicefrac{1}{2}}(X_{ij})_{i, j \in [N]}$ over some probability space $(\Omega, \mathcal{F}, \mathbb{P})$ is a \textit{GUE matrix} if its entries satisfy the following conditions: Each $X_{ij}$ is a standard Gaussian random variable on $(\Omega, \mathcal{F}, \mathbb{P})$, which is complex for $i \neq j$ and real for $i=j$, and $\{X_{ij}: i \leq j\}$ are independent. Recall that Wigner's semicircle law applies to GUE matrices and that tuples of independent GUE matrices of dimension $N \times N$ almost surely strongly converge to tuples of free standard semicircular random variables in some $C^*$-probability space as $N \rightarrow \infty$. We refer to \cite{Mingo2017} for more details on the last-mentioned convergence results. 
	
	Having introduced the necessary notation, we can formulate our result on self-normalized sums of independent GUE matrices. 
	\begin{prop} \label{example - main theorem}
		For all choices of  $n, N \in \mathbb{N}$, let $\smash{X_{n,N} = \big(X^{(N)}_1, \dots, X^{(N)}_n\big)}$ be an $n$-tuple of independent GUE matrices of dimension $N \times N$ over some probability space $(\Omega, \mathcal{F}, \mathbb{P})$. Define
		\begin{align*}
			S_{n,N} := \sum_{i=1}^{n} X^{(N)}_i, \qquad V_{n,N}^2  :=  \sum_{i=1}^{n}\big( X^{(N)}_i\big)^2.
		\end{align*} 
		Then, there exists $N_0 \in \mathbb{N}$ such that for $\mathbb{P}$-a.e.\@ $\omega \in \Omega$ the self-normalized sum $U_{n,N}(\omega)$ given by 
		\begin{align*}
			U_{n,N}(\omega) := \left(V_{n,N}^{2}(\omega)\right)^{-\nicefrac{1}{4}}S_{n,N}(\omega) \left(V_{n,N}^{2}(\omega)\right)^{-\nicefrac{1}{4}}
		\end{align*}
		is well-defined in $M_N(\mathbb{C})$ for all $N \geq N_0$, $n \in \mathbb{N}$. 
		Moreover, for $\mathbb{P}$-a.e.\@ $\omega \in \Omega $ and all $z \in \mathbb{C}^+$, we have
		\begin{align*}
			\lim_{n\rightarrow \infty}  \limsup_{N \rightarrow \infty} \left\vert \tr_N\left( (z- U_{n,N}(\omega))^{-1} \right) - G_\omega(z) \right\vert = 0, 
		\end{align*}
	where $G_\omega(z)$ denotes the Cauchy transform of Wigner's semicircle law $\omega$ (not to be confused with the element $\omega \in \Omega$) evaluated at $z \in \mathbb{C}^+$.
	\end{prop}
	
	Before we provide the proof of the proposition above, let us make a few comments: First, an application of the results established in \cref{SNS section bounded proofs} is not possible since the GUE matrices $\smash{X_1^{(N)}, \dots, X_n^{(N)}}$ are neither free nor elements in a $C^*$-probability space. Moreover, Yin's strong convergence result \cite[Theorem 1.2]{Yin2018} as used in \cref{SNS bounded Remark Yin} does not imply \cref{example - main theorem} either due to the consideration of the double limit. 
	Second, the sole purpose of \cref{example - main theorem} and its proof is to show how some of the ideas of the intuitive approach to (bounded) self-normalized sums are still applicable in the setting of independent GUE matrices. In particular, the proof of \cref{example - main theorem} is dominated by the free probabilistic approach to random matrices. 
	
	\begin{proof} [Proof of \cref{example - main theorem}]
	We proceed in three steps: In the first step, we verify that $U_{n,N}$ is $\mathbb{P}$-a.s.\@ well-defined for all $n$ and all sufficiently large $N$. 
		The second step, non-rigorously speaking, consists of showing that $\smash{V_{n,N}^2}$ and $\smash{V_{n,N}^{-2}}$ are $\mathbb{P}$-a.s.\@ close to the identity matrix $I_N \in M_N(\mathbb{C})$ in the operator norm $\Vert \cdot \Vert_N$ for large $N$ and $n$. % Let us remark that -- in contrast to the setting of  self-normalized sums of free bounded random variables as considered in \cref{SNS section bounded} -- the closeness between $V_{n,N}^2$ and $I_N$ is not relevant for the first step above.
		In the third step, we prove the claimed convergence by comparing $\smash{\tr_N((z- U_{n,N})^{-1})}$ to $\smash{\tr_N((z- S_{n,N})^{-1})}$ for any $z \in \mathbb{C}^+$ and by using Wigner's semicircle law for GUE matrices.
		
		Before we start with the first step, let us introduce some notation.
		As before, we agree on considering the normalized versions of $S_{n,N}$ and $V_{n,N}^2$, i.e.\@ from now on, let 
		\begin{align*}
			S_{n,N} = \frac{1}{\sqrt{n}}\sum_{i=1}^{n} X^{(N)}_i, \qquad V_{n,N}^2  = \frac{1}{n} \sum_{i=1}^{n} \big(X^{(N)}_i\big)^2.
		\end{align*} 
		Define the non-commutative polynomial $P_n$ in $n$ formal variables by $P_n(x_1, \dots, x_n) := \sum_{i=1}^{n} x_i^2$.	Moreover, set $r_n(x_1, \dots, x_n) := (P_n(x_1, \dots, x_n))^{-1}$ and observe that $r_n$ is a non-degenerate non-commutative rational expression. 
		Lastly, after passing over to a countable union of null sets if necessary, we find a null set $M_0 \subset \Omega$ such that for all $\omega \not \in M_0$ and all $N,n \in \mathbb{N}$, $i \in [n]$, we have 
		\begin{align*}
			\Big(X_i^{(N)}(\omega) \Big)^* = X_i^{(N)}(\omega) \in M_N(\mathbb{C}). 
		\end{align*}
		
		Now, we begin with the first step. Clearly, it suffices to show that $\smash{V_{n,N}^2}$ is $\mathbb{P}$-a.s.\@ invertible for all $n$ and sufficiently large $N$. For this purpose, we use \cite[Theorem 18]{Collins2022} dealing with evaluations of certain rational expressions in random matrices that are distributed according to an absolutely continuous probability measure. Let us go into detail here: Fix $n \in \mathbb{N}$ and $i \in [n]$. Applying the last-cited theorem to the rational expression $r_1$ and the GUE matrix $\smash{X_i^{(N)}}$, we find $N_0 \in \mathbb{N}$ and a null set $M_1 \subset \Omega$ -- both independent of $i$ and $n$ --  such that for all $\omega \not \in M_1$ and all $N \geq N_0$ the matrix
		\begin{align*}
		 r_1 \left(X_i^{(N)}(\omega)\right)  = 	\left(X_i^{(N)}(\omega) \right)^{-2} 
		\end{align*}
		exists in $M_N(\mathbb{C})$. In particular, it follows that $\big(X_i^{(N)}(\omega)\big)^2$ is positive definite for all  $\omega \not \in M_0 \cup M_1$ and $N, n,$ and $i$ as above.  
	Together with
		\begin{align*}
			P_n (X_{n,N}(\omega))    = \sum_{i=1}^{n}	P_1\left(X_i^{(N)}(\omega)\right),
			\end{align*}
		we deduce that $V_{n,N}^2(\omega) = n^{-1}P_n(X_{n,N}(\omega))$ is positive definite and thus invertible for all $\omega \not \in M_0 \cup M_1$, $N \geq N_0$, and $n \in \mathbb{N}$. 
		For completeness, let us remark that it is possible to apply \cite[Theorem 18]{Collins2022} to $r_n$ and $X_{n,N}$ directly, i.e.\@ the detour via $r_1$ can be avoided. However, by this  approach, it is not immediate that the lower bound on the matrix dimension $N$ is independent of $n$.
		
		Before we continue with the second step, let us do some preparatory work. For this, fix $n \in \mathbb{N}$ and let $(s_1, \dots, s_n)$ be an $n$-tuple of free standard semicircular random variables in some $C^*$-probability space $(\mathcal{A}, \varphi)$ with faithful tracial  functional $\varphi$, unit $1=1_{\mathcal{A}}$, and norm $\Vert \cdot \Vert_{\mathcal{A}}$. Making use of some information on \textit{free Poisson elements} taken from \cite[Propositions 12.11 and 12.13]{Nica2006}, our goal is to analyze the random variables $n^{-1}P_n(s_1, \dots, s_n)$ and $nr_n(s_1, \dots, s_n)$. Note that $nr_n(s_1, \dots, s_n)$ a priori does not need to exist in $\mathcal{A}$. Recall that $s_i^2$ is a free Poisson element of rate and jump size both equal to $1$ for any $i \in [n]$. In particular, we have $\kappa_m(s_i^2) =1$ for all $m \in \mathbb{N}, i \in [n]$, where  $\kappa_m(\cdot)$ denotes the \textit{$m$-th free cumulant} as defined in \cite[Definition 11.3]{Nica2006}. By freeness of $s_1^2, \dots, s_n^2$ and vanishing of mixed cumulants, it follows
		\begin{align*}
			\kappa_m\left(  \frac{P_n(s_1, \dots, s_n)}{n}\right) = \kappa_m \left( \frac{1}{n} \sum_{i=1}^n s _i^2 \right) = n \left( \frac{1}{n} \right)^m =  \frac{1}{n^{m-1}}
		\end{align*}
		for all $m \in \mathbb{N}$ implying that $n^{-1}P_n(s_1, \dots, s_n)$ is a free Poisson element with rate $n$ and jump size $n^{-1}$. Since $n^{-1}P_n(s_1, \dots, s_n)$ is normal and $\varphi$ is faithful, we get
		\begin{align*}
			\Sp\left( \frac{P_n(s_1, \dots, s_n)}{n} \right) = \left[ 1- \frac{2}{\sqrt{n}} + \frac{1}{n}, 1+ \frac{2}{\sqrt{n}} + \frac{1}{n}\right] \subset [0, \infty)
		\end{align*}
		leading to
		\begin{align} \label{SNS RMT Bsp bound P -1}
			\left \Vert \frac{P_n(s_1, \dots, s_n)}{n} - 1 \right \Vert_{\mathcal{A}} = \sup\left\{ \vert z-1 \vert: z \in 	\Sp\left( \frac{P_n(s_1, \dots, s_n)}{n} \right) \right\} = \frac{2}{\sqrt{n}} + \frac{1}{n}.
		\end{align}
		Under the additional assumption $n \geq 20$, we obtain $ \Vert n^{-1}P_n(s_1, \dots, s_n) - 1 \Vert_{\mathcal{A}} \leq \nicefrac{1}{2}$. Then, \cref{Lemma Lin} yields that $n^{-1}P_n(s_1, \dots, s_n)$ is invertible in $\mathcal{A}$ with inverse $nr_n(s_1, \dots, s_n) \in \mathcal{A}$ satisfying
		\begin{align} \label{SNS RMT Bsp bound r -1}
			\left\Vert n r_n(s_1, \dots, s_n) -1 \right\Vert_{\mathcal{A}} \leq \frac{ \left\Vert n^{-1}P_n(s_1, \dots, s_n) - 1 \right\Vert_{\mathcal{A}}}{1 - \left \Vert n^{-1}P_n(s_1, \dots, s_n) - 1 \right \Vert_{\mathcal{A}}} \leq	\frac{4}{\sqrt{n}} + \frac{2}{n}
		\end{align}
		for all $n \geq 20.$
		
		Now, we turn to the second step of the proof. Since $nr_n(s_1, \dots, s_n)$ is defined in $\mathcal{A}$, we can apply \cite[Corollary 1.3]{Yin2018} to the rational expressions $nr_n(x_1, \dots, x_n) -1$ and $n^{-1}P_n(x_1, \dots, x_n) -1$ for any $n \in \mathbb{N}$. Consequently, we find a null set $M_2 \subset \Omega$ -- again after passing over to a countable union of null sets if necessary -- such that for all $\omega \not \in \cup_{i=0}^2 M_i$ and all $ N \geq N_0$, $n \in \mathbb{N}$,  we have
		\begin{align} \label{SNS RMT Bsp Yin}
			\begin{split}
				\lim_{N \rightarrow \infty} \big\Vert V_{n,N}^{-2}(\omega) - I_N \big\Vert_N & =	 \left\Vert nr_n(s_1, \dots, s_n) -1 \right \Vert_{\mathcal{A}}, \\ 
				\lim_{N \rightarrow \infty} \big\Vert V_{n,N}^{2}(\omega) - I_N \big\Vert_N &= \left\Vert \frac{P_n(s_1, \dots, s_n)}{n} - 1 \right\Vert_{\mathcal{A}}.
			\end{split}
		\end{align}
		
		Finally, let us continue with the third step. Here, we argue realization-wise. Thus, for the following calculations, fix $\smash{\omega \not \in \cup_{i=0}^2 M_i}$ as well as  $N \geq N_0$, $n \in \mathbb{N}$, and $z \in \mathbb{C}^+.$ Because of $\omega \not \in M_0$, the realizations $\smash{X_1^{(N)}(\omega), \dots,X_n^{(N)}(\omega) }$ are self-adjoint random variables in the $C^*$-probability space $\smash{(M_N(\mathbb{C}), \tr_N)}$. In particular,  we can reuse some of the ideas presented in the proof of \cref{bounded_main_non_id}.
		
		We already know that $\smash{\Vert (z-U_{n,N}(\omega))^{-1} \Vert_N \leq (\Im z)^{-1}}$ and  $\smash{\Vert (z-S_{n,N}(\omega))^{-1} \Vert_N \leq (\Im z)^{-1}}$ hold. Note that $zV_{n,N}(\omega) - S_{n,N}(\omega)$ is invertible in $M_N(\mathbb{C})$ with inverse 
		\begin{align*}
			\left(zV_{n,N}(\omega) - S_{n,N}(\omega)\right)^{-1} = V_{n,N}^{-\nicefrac{1}{2}}(\omega)(z-U_{n,N}(\omega))^{-1} V_{n,N}^{-\nicefrac{1}{2}}(\omega).
		\end{align*}
		Thus, we obtain 
		\begin{align*}
			\left\Vert \left(zV_{n,N}(\omega) - S_{n,N}(\omega)\right)^{-1} \right\Vert_N \leq \frac{  \left \Vert V_{n,N}^{-1}(\omega) \right\Vert_N}{\Im z}.
		\end{align*}
		Since $\smash{V_{n,N}^{-1}(\omega)}$ is positive definite, all eigenvalues of $\smash{V_{n,N}^{-1}(\omega) + I_N}$ are larger than 1, implying that the matrix $\smash{V_{n,N}^{-1}(\omega) + I_N}$ is invertible in $M_N(\mathbb{C})$ with 
		\begin{align*}
			\left \Vert ( V_{n,N}^{-1}(\omega) + I_N )^{-1}  \right\Vert_N  \leq 1.
		\end{align*}
		Together with $\smash{V_{n,N}^{-1}(\omega) - I_N = ( V_{n,N}^{-1}(\omega) + I_N )^{-1}(V_{n,N}^{-2}(\omega) - I_N)}$, it follows
		\begin{align*}
			\left\Vert 	V_{n,N}^{-1}(\omega) - I_N \right \Vert_N \leq \left \Vert \big( V_{n,N}^{-1}(\omega) + I_N \big)^{-1} \right \Vert_N \left\Vert  V_{n,N}^{-2}(\omega) - I_N \right \Vert_N \leq \left \Vert  V_{n,N}^{-2}(\omega) - I_N \right \Vert_N,
		\end{align*}
		which leads to
		\begin{align*}
			\left \Vert V_{n,N}^{-1} (\omega) \right \Vert_N \leq  1 + \left\Vert V_{n,N}^{-2}(\omega) - I_N \right\Vert_{N}.
		\end{align*}
		Moreover, by using the same arguments as above, we get
		\begin{align*}
			\left \Vert V_{n,N}(\omega) - I_N \right \Vert_N \leq \left \Vert \left( V_{n,N}(\omega) + I_N \right)^{-1} \right \Vert_N \left \Vert  V_{n,N}^2(\omega) - I_N \right \Vert_N \leq \left\Vert  V_{n,N}^2(\omega) - I_N \right \Vert_N.
		\end{align*}
		Now, we write
		\begin{align*}
			& \tr_N \left( \left(z - U_{n,N}(\omega)\right)^{-1} - \left(z -S_{n,N}(\omega)\right)^{-1} \right) \\ & \qquad =\tr_N \left( (V_{n,N}(\omega) - I_N)\left(zV_{n,N}(\omega) - S_{n,N}(\omega)\right)^{-1} \right)  \\& \qquad \qquad  + \tr_N\left( \left(zV_{n,N}(\omega) - S_{n,N}(\omega) \right)^{-1}- (z-S_{n,N}(\omega))^{-1} \right) 
			\\ & \qquad =\tr_N \left( (V_{n,N}(\omega) - I_N)\left(zV_{n,N}(\omega) - S_{n,N}(\omega)\right)^{-1} \right)  \\ & \qquad \qquad + \tr_N\left( \left(zV_{n,N}(\omega)- S_{n,N}(\omega) \right)^{-1}z(I_N - V_{n,N}(\omega)) (z-S_{n,N}(\omega))^{-1} \right).
		\end{align*}
		The first term in the last sum above admits the estimate 
		\begin{align*}
			\left \vert  \tr_N \left( (V_{n,N}(\omega) - I_N)\left(zV_{n,N}(\omega) - S_{n,N}(\omega)\right)^{-1} \right) \right \vert 
			& \leq  \big\Vert  V_{n,N}(\omega) - I_N \big \Vert_N \big \Vert \left(zV_{n,N}(\omega) - S_{n,N}(\omega)\right)^{-1}  \big \Vert_N \\ &\leq \frac{\big\Vert  V_{n,N}^2(\omega) - I_N \big \Vert_N \left(1+  \left\Vert V_{n,N}^{-2}(\omega) - I_N \right\Vert_{N}\right)}{\Im z},
		\end{align*}
		whereas the second term can be bounded by 
		\begin{align*}
			&\left \vert \tr_N\left( \left(zV_{n,N}(\omega)- S_{n,N}(\omega)\right)^{-1}z(I_N - V_{n,N}(\omega)) (z-S_{n,N}(\omega))^{-1} \right) \right \vert \\ & \,\,\, \leq \vert z \vert \left\Vert V_{n,N}(\omega) - I_N \right\Vert_N \left \Vert \left(zV_{n,N}(\omega) -S_{n,N}({\omega})\right)^{-1}\right \Vert_N  \left\Vert \left(z-S_{n,N}(\omega)\right)^{-1} \right\Vert_N \\ & \,\,\, \leq \frac{\vert z \vert}{(\Im z)^2} \left \Vert V_{n,N}^2(\omega) - I_N \right \Vert_N \left( 1 + \big\Vert V_{n,N}^{-2}(\omega) - I_N \big\Vert_N \right).
		\end{align*}
		It is immediate that the last two inequalities combined with \eqref{SNS RMT Bsp bound P -1}, \eqref{SNS RMT Bsp bound r -1}, and \eqref{SNS RMT Bsp Yin} imply 
		\begin{align*}
			\lim_{n \rightarrow \infty} \limsup_{N \rightarrow \infty} \left \vert  \tr_N \left( \left(z - U_{n,N}(\omega)\right)^{-1} - \left(z -S_{n,N}(\omega)\right)^{-1} \right)  \right \vert = 0.
		\end{align*}
		Since $S_{n,N}$ is a GUE matrix for any $n, N \in \mathbb{N}$, Wigner's semicircle law is applicable. Hence, there exists a null set $M_3 \subset \Omega$ such that for all $\omega \not \in M_3$, $z \in \mathbb{C}^+$, and $n \in \mathbb{N}$, we have 
		\begin{align*}
			\lim_{N \rightarrow \infty} \tr_N\left(\left(z-S_{n,N}(\omega)\right)^{-1}\right) = G_\omega(z).
		\end{align*}
	
		Finally, we conclude: There exists a null set $M \subset \Omega$, namely $M := \cup_{i=0}^3 M_i$, such that for all $\omega \not \in M$ and all $z \in \mathbb{C}^+$, we have
		\begin{align*}
			\lim_{n\rightarrow \infty}  \limsup_{N \rightarrow \infty} \left\vert \tr_N\left( (z- U_{n,N}(\omega))^{-1} \right) - G_\omega(z) \right\vert = 0.
		\end{align*}
	\end{proof}

	\section{Unbounded self-normalized sums: Proofs of Theorem \texorpdfstring{\ref{unbounded_main_BE}}{1.4} and Corollary \texorpdfstring{\ref{unbounded main id}}{1.5}} \label{SNS sec unbounded}
	This section is devoted to the analysis of self-normalized sums of free self-adjoint unbounded operators that are affiliated with a von Neumann algebra. We start with proving \cref{unbounded_main_BE}, from which the claims in \cref{unbounded main id} dealing with the special case of identical distributions can be derived easily. 
	
	Throughout this section, we work in the following setting: Let $(\mathcal{A}, \varphi)$ be a $W^*$-probability space acting on some Hilbert space $\mathcal{H}$ with tracial functional $\varphi$. Recall that the identity operator in $B(\mathcal{H})$ is denoted by $1$, whereas the operator norm is given by $\Vert \cdot \Vert$. Fix a sequence $\smash{(X_i)_{i \in \mathbb{N}} \subset \mathcal{L}^2(\mathcal{A}) \subset \Aff(\mathcal{A})}$ of free self-adjoint random variables with $\varphi(X_i) = 0$ and $\varphi(X_i^2) = \sigma_i^2$ for all $i \in \mathbb{N}$. Let $B_n^2$ be defined as in \cref{bounded_main_non_id} and assume that  $B_n^2>0$ as well as Lindeberg's condition in \eqref{Lindeberg condition} are satisfied. In agreement with the notation introduced in Section \ref{SNS section bounded}, we consider the normalized versions of $S_n$ and $V_n^2$ as given in \eqref{bounded - normalized versions of S_n, V_n^2}. In the case of existence, $ \mu_n$ denotes the analytic distribution of the self-normalized sum $\smash{U_n = V_n^{-\nicefrac{1}{2}}S_n V_n^{-\nicefrac{1}{2}}}$ and $F_n$ and $G_n$ are its distribution function and Cauchy transform.  Lastly, $G_{S_n}$ denotes the Cauchy transform of the analytic distribution $\mu_{S_n}$ of $S_n.$  \\
	
	Before we start with some preparations for the proof of \cref{unbounded_main_BE}, let us briefly comment on the general idea. On a superficial level, we proceed similarly to the proof of \cref{bounded_main_non_id}: First, we show that $U_n$ is well-defined in $\Aff(\mathcal{A})$ for large $n$; compare to \cref{unbounded_V_n is invertible}. Second, in order to derive the claimed rate of convergence of $\mu_n$ to $\omega$, we argue by a version of Bai's inequality and bound the resulting differences between the Cauchy transforms $G_n$ and $G_\omega$ by making use of the closeness of $V_n^2$ to $1$ and the free CLT. Third, the convergence claim at the end of \cref{unbounded_main_BE} follows immediately from calculations done in the second step. As we will see in the course of this section, the explicit realizations of the first two steps differ strongly from the ones given in \cref{SNS section bounded}. \\
	
	We begin by showing that the self-normalized sum $U_n$ is well-defined in $\Aff(\mathcal{A})$ for sufficiently large $n$. For this, we have to verify that $V_n^2$ is invertible in $\Aff(\mathcal{A})$ for large $n$.
	The following lemma provides a sufficient condition for a self-adjoint operator in $\Aff(\mathcal{A})$ to be invertible in $\Aff(\mathcal{A})$. Due to its generality, the corresponding statement could also be included in \cref{SNS section preliminaries affiliated}. However, we chose to give the lemma here for better comparability of the proofs of the invertibility of $V_n^2$ in the bounded and unbounded case; compare to  \cref{SNS remark comparison invertibility}.
	
	\begin{lem} \label{unbounded - V_n^2 invertible lemma}
		Let $T \in \Aff(\mathcal{A})$ be self-adjoint with spectral measure $E_T$ satisfying $E_T(\{ 0\}) = 0$. Then, $T$ is invertible in $\Aff(\mathcal{A})$. 
		\begin{proof} 
			Due to $E_T(\{0\}) = 0$, the identity map is $E_T$-a.s.\@ non-zero on $\mathbb{R}$ and we can make use of \cite[Theorem 5.9]{Schmuedgen2012}. It follows that $T$ is invertible with $T^{-1} = f(T)$ for $f: \mathbb{R} \rightarrow \mathbb{R} \cup \{\infty\}$, $f(x) := \nicefrac{1}{x}$ for $x \neq 0$, $f(0):= \infty$. Since $f$ is $E_T$-a.s.\@ finite and due to $T \in \Aff(\mathcal{A})_{\text{sa}}$, we obtain $T^{-1} = f(T) \in \Aff(\mathcal{A})$; compare to \cref{SNS section preliminaries affiliated}. 
		\end{proof}
	\end{lem}
	
		Let us briefly add some remarks to the last lemma.
		
	\begin{rem} \label{unbounded bem invertibility lemma}
		\begin{enumerate}[(i)]
			\item For $T$ as in \cref{unbounded - V_n^2 invertible lemma}, the condition $E_T(\{ 0\}) = 0$ is equivalent to $\mu_T(\{0\}) = 0$ with $\mu_T$ being the analytic distribution of $T$. For self-adjoint $T$ (not necessarily in $\Aff(\mathcal{A})$), the equation $E_T(\{0\}) = 0$ is equivalent to $T$ being injective; compare to \cite[Proposition 5.10]{Schmuedgen2012}.
			\item As shown in  \cite[Lemma 2.6]{Reich1998}, any injective operator in $\Aff(\mathcal{A})$ is invertible in $\Aff(\mathcal{A})$. The proof of this statement is based on the polar decomposition of closed densely defined operators, functional calculus, and the faithfulness of $\varphi$. Tembo \cite[Lemma 3.2]{Tembo2009} proved that any injective normal operator in $\Aff(\mathcal{A})$ is invertible in $\Aff(\mathcal{A})$ without the use of functional calculus. 
		\end{enumerate}
	\end{rem}
	
	Later, we would like to apply \cref{unbounded - V_n^2 invertible lemma} to the sum $V_n^2$. In view of \cref{unbounded bem invertibility lemma}, it suffices to gather some information on possible atoms of the free additive convolution
	$\smash{\mu_{V_n^2} = \mu_{\nicefrac{X_1^2}{B_n^2}} \boxplus \cdots \boxplus\mu_{\nicefrac{X_n^2}{B_n^2}}}.$ For this purpose, we use the following well-known regularity result for free additive convolutions proven by Bercovici and Voiculescu \cite[Theorem 7.4]{Bercovici1998}.
	
	\begin{prop}[] \label{atom}
		Let $\nu_1$ and $\nu_2$ be probability measures on $\mathbb{R}$ and let $\gamma \in \mathbb{R}.$ The following are equivalent:
		\begin{enumerate}[(i)]
			\item $\gamma$ is an atom of $\nu_1 \boxplus \nu_2$.
			\item There exist atoms $\alpha, \beta \in \mathbb{R}$ of $\nu_1$, $\nu_2$, respectively, with $\gamma = \alpha + \beta$ and $\nu_1(\{ \alpha \}) + \nu_2(\{ \beta\})>1$.
		\end{enumerate}
		If $(ii)$ holds, then $(\nu_1 \boxplus \nu_2)(\{\gamma\}) = \nu_1(\{ \alpha \}) + \nu_2(\{ \beta\}) -1$.
	\end{prop}
	
	Finally, we are able to prove that $V_n^2$ is invertible in $\Aff(\mathcal{A})$ for sufficiently large $n$. As a byproduct of our proof, we obtain that the same conclusion holds true for the random variable $S_n$. Let us note that the invertibility of $S_n$ in $\Aff(\mathcal{A})$ is not very surprising considering \cref{unbounded - V_n^2 invertible lemma}, \cref{unbounded bem invertibility lemma}, and the phenomenon of superconvergence in the free CLT.

	\begin{lem}\label{unbounded_V_n is invertible}
		There exists $n_0 \in \mathbb{N}$ such that $V_n^2, V_n, V_n^{\nicefrac{1}{2}}$, and $S_n$ are invertible in $\Aff(\mathcal{A})$ for all $n \geq n_0.$
		\begin{proof}
			Let us begin by proving that the analytic distribution $\mu_{S_n}$ has no atoms for sufficiently large $n.$
			It is well-known that Lindeberg's condition, which we assumed to hold at the beginning of this section, implies 
			\begin{align*}
				\lim_{n \rightarrow \infty}\frac{  \max_{i \in [n]} \sigma_i^2}{B_n^2} = 0.
			\end{align*}
			Chebyshev's inequality yields
			\begin{align*}
				\max_{i \in [n]} \mu_{\nicefrac{X_i}{B_n}}(\{ x: \vert x \vert \geq \varepsilon \}) \leq \varepsilon^{-2} \max_{i \in [n]} \varphi\left( \frac{X_i^2}{B_n^2}\right) = \varepsilon^{-2}  \frac{\max_{i \in [n]}\sigma_i^2}{B_n^2}  \rightarrow 0 
			\end{align*}
			as $n \rightarrow \infty$ for any $\varepsilon>0$, i.e.\@ the family $\smash{\{\mu_{\nicefrac{X_i}{B_n}}: n \in \mathbb{N}, i \in [n]\}}$ is an infinitesimal array of probability measures. Moreover, according to the free Lindeberg CLT, we have $\smash{\mu_{S_n} = \mu_{\nicefrac{X_1}{B_n}} \boxplus \dots \boxplus  \mu_{\nicefrac{X_n}{B_n}} \Rightarrow \omega}$ as $n \rightarrow \infty.$ The last two observations allow to apply the superconvergence result for free additive convolutions in \cite[Corollary 2.4]{Bercovici2022} implying that $\mu_{S_n}$ is absolutely continuous in a neighborhood of $[-1,1]$ for sufficiently large $n$, say $n \geq n_0.$ In particular, we get $\mu_{S_n}(\{0\}) = 0$ for those $n$. By \cref{unbounded - V_n^2 invertible lemma}, $S_n$ is invertible in $\Aff(\mathcal{A})$ for all $n \geq n_0$.
			
			In order to verify the claims concerning $V_n^2$, $V_n$, and $\smash{V_n^{\nicefrac{1}{2}}}$, we argue similarly: Assume that $0$ is an atom of $\mu_{V_n^2}$. Then, by induction on \cref{atom}, we find atoms $x_i \in \mathbb{R}$ of $\mu_{\nicefrac{X_i^2}{B_n^2}}$ such that 
			\begin{align*}
				0 = \sum_{i=1}^{n} x_i  \qquad \text{and} \qquad \sum_{i=1}^{n} \mu_{\nicefrac{X_i^2}{B_n^2}}(\{ x_i\}) > n-1 
			\end{align*}
			hold. 
			By positivity of $\nicefrac{X_i^2}{B_n^2}$ for all $i \in [n]$, we must have $x_i = 0$ for all $i \in [n]$. Making use of functional calculus, it follows $\smash{	0< \mu_{\nicefrac{X_i^2}{B_n^2}}(\{0 \}) =\mu_{\nicefrac{X_i}{B_n}}(\{0\})}$
			for all $ i \in [n]$.
			In particular, \cref{atom} -- this time in the reverse direction -- yields that $0$ has to be an atom of $\mu_{S_n}$, which is a contradiction whenever $n \geq n_0$ holds. Hence, we obtain $\smash{0 = \mu_{V_n^2}(\{0\}) =  \mu_{V_n}(\{0\})  =  \mu_{\smash{V_n^{\nicefrac{1}{2}}}}(\{0\})}$ for $n \geq n_0$. 
		\end{proof}
	\end{lem}
	
	Let us briefly compare the proofs of the invertibility of $V_n^2$ in the bounded and unbounded case; compare to \cref{bounded_Vn is invertible,unbounded_V_n is invertible}.
	\begin{rem}  \label{SNS remark comparison invertibility}
		\begin{enumerate}[(i)]
			\item In the bounded case, i.e.\@ in \cref{bounded_Vn is invertible}, the invertibility of $V_n^2$ was derived with the help of a Neumann series type argument with respect to the norm defined on the underlying $C^*$-algebra combined with the fact that $V_n^2$ is close to $1$ in that norm. This approach fails in the unbounded case since the operator norm of $V_n^2$ may be infinite. Additionally, in the proof of \cref{unbounded_V_n is invertible}, we did not use the closeness of the unbounded operator $V_n^2$ to the identity $1.$
			\item The proof of the invertibility of $V_n^2$ in the unbounded case, compare to \cref{unbounded_V_n is invertible}, heavily relies on the fact that $V_n^2$ is a linear operator (and not just an abstract object in some $C^*$-algebra) allowing us to talk about injectivity and spectral measures. In particular, $V_n^2$ being invertible is not tied to the condition that $0$ is not contained in its spectrum and we can make sense of $V_n^{-2}$ as a linear operator whenever the premise of \cref{unbounded - V_n^2 invertible lemma} is satisfied, i.e.\@ whenever $V_n^2$ is injective.
			A statement similar to that in \cref{unbounded - V_n^2 invertible lemma} is not possible in the context of $C^*$-probability spaces, even in view of the GNS representation.
			% A natural analog of the condition in \cref{unbounded - V_n^2 invertible lemma} would be $\mu_x (\{0\}) =0$ with $\mu_x$ denoting the analytic distribution of a random variable $x$ in the corresponding $C^*$-algebra. However, in general, $\mu_x (\{0\}) =0$ does not imply $0 \not \in \Sp(x)$, which is equivalent to $x$ being invertible.
		\end{enumerate}
	\end{rem}
	
	In the following, assume that $n \geq n_0$ holds with $n_0$ taken from the preceding lemma.  Then, the self-normalized sum $U_n$ is well-defined in $\Aff(\mathcal{A}).$ Before we pass over to the complete proof of \cref{unbounded_main_BE}, let us collect some information on the integrability of $U_n.$ We start by recalling the following lemma taken from \cite[Lemma 2.1]{Bikchentaev2012}.
	
	\begin{lem} \label{Bikchentaev square sum smallelr sum of squares}
		Let $T_1, \dots, T_k \in \Aff(\mathcal{A})$ and $\lambda_1, \dots, \lambda_k >0$ with $\sum_{i=1}^{k} \lambda_i \leq 1$. Then, we have 
		\begin{align*}
			\left \vert \sum_{i=1}^k \lambda_i T_i\right \vert^2 \leq \sum_{i=1}^{k} \lambda_i\vert T_i \vert^2.
		\end{align*} 
	\end{lem}
	
	Clearly, the last lemma implies $0 \leq \vert S_n \vert^2 \leq n V_n^2$. With the help of this inequality, we can prove that the self-normalized sum $U_n$ satisfies $\vert U_n \vert \leq \sqrt{n}$ and is contained in $\mathcal{L}^p(\mathcal{A})$ for all $0<p\leq \infty.$ Let us remark that these statements can be interpreted as a consequence of some kind of regularizing effect of the self-normalization. By Cauchy's inequality, a comparable regularizing effect is observable in the case of self-normalized sums in classical probability theory.
	\begin{lem} \label{unbounded Un smaller sqrt n}
		We have $U_n^2 \leq n$ and thus obtain $\vert U_n \vert \leq \sqrt{n}$. It follows $U_n \in \mathcal{A}$ and $\Vert U_n \Vert_p \leq \sqrt{n}$ for any $p\in (0, \infty]$.
		\begin{proof}
			From the proof of \cref{unbounded_V_n is invertible}, we get $\mu_{\vert S_n \vert}(\{0\})=0$ and thus conclude that $\vert S_n \vert$ is invertible in $\Aff(\mathcal{A})$. Using \cref{Bikchentaev square sum smallelr sum of squares} and some of the properties of the partial order $\smash{\leq}$ given in \cref{SNS section preliminaries affiliated}, we obtain $0 \leq \vert S_n \vert \leq \sqrt{n}V_n$, which yields $0 \leq V_n^{-1} \leq \sqrt{n} \vert S_n \vert^{-1}$. The functions $x \mapsto x \vert x \vert^{-1}x = x^2 \vert x \vert^{-1}$ with $0$ being sent to $0$ and $x \mapsto \vert x \vert$ ($E_{S_n}$-a.s.\@) agree on $\mathbb{R}$. %(due to $	0 = \mu_{S_n}(\{0\}) $) 
			Hence, by functional calculus,  we deduce $S_n \vert S_n \vert^{-1} S_n = \vert S_n \vert$.
			Finally, it follows
			\begin{align*}
				\vert U_n \vert^2 = U_n^2 = U_n^* U_n &= V_n^{-\nicefrac{1}{2}} S_n V_n^{-1} S_n V_n^{-\nicefrac{1}{2}} = \left(S_n V_n^{-\nicefrac{1}{2}} \right)^*  V_n^{-1} \left(S_n V_n^{-\nicefrac{1}{2}}\right) \leq  \sqrt{n}V_n^{-\nicefrac{1}{2}}S_n \vert S_n \vert^{-1} S_nV_n^{-\nicefrac{1}{2}} \\ &=
				\sqrt{n}V_n^{-\nicefrac{1}{2}} \vert S_n \vert V_n^{-\nicefrac{1}{2}} \leq n,
			\end{align*}
			which implies $\vert U_n \vert \leq \sqrt{n}$. In particular, as stated in \cref{SNS section preliminaries affiliated}, we obtain $\mathcal{H} = D(\sqrt{n}) \subset D(\vert U_n \vert)$. Combining this with the closed graph theorem and the fact that $U_n$ is affiliated with $\mathcal{A}$, we get $U_n \in \mathcal{A}$.
			In order to prove the last claim, recall that we have 
			\begin{align*}
				\Vert U_n \Vert_p =	\left( \int_{0}^1 \mu_t(U_n)^p dt \right)^{\nicefrac{1}{p}} \in [0, \infty]
			\end{align*}
			for any $p \in (0, \infty).$
			Together with $\mu_t(U_n) = \mu_t(\vert U_n \vert) \leq \sqrt{n} \mu_t(1) \leq \sqrt{n}$ holding for all $t>0$, we arrive at $\Vert U_n \Vert_p \leq \sqrt{n}$ for $p$ as above. For $p= \infty$, the claimed inequality $\Vert U_n \Vert \leq \sqrt{n}$ is immediate from $\lim_{t \searrow 0} \mu_t( \vert U_n \vert) = \Vert U_n \Vert.$ 
		\end{proof}
	\end{lem}
	
	According to the last lemma, we have $\Vert U_n \Vert_2 \leq \sqrt{n}$ whenever the underlying sequence of random variables $(X_i)_{i \in \mathbb{N}}$ is in $\mathcal{L}^2(\mathcal{A})$. The following lemma shows that under stronger moment assumptions, we obtain a new bound on $\Vert U_n \Vert_2$. This bound contains a summand of the form $ \sqrt{n}\Vert V_n^2 - 1 \Vert_2$, which -- in view of our expectation that $V_n^2$ is close to $1$ -- will turn out to be advantageous in the proof of \cref{unbounded_main_BE}.
	
	\begin{lem}\label{second moment SNS}
		Assume that $(X_i)_{i \in \mathbb{N}} \subset \mathcal{L}^4(\mathcal{A})$ holds. Then, we have
		\begin{align*}
			\left\Vert U_n \right \Vert_2  \leq  2 +3\sqrt{2n} \big\Vert V_n^2 - 1 \big\Vert_2  \qquad \text{and} \qquad 	\big\Vert S_n V_n^{-\nicefrac{1}{2}}\big \Vert_2 \leq \sqrt{2} + \sqrt{n} \left\Vert V_n^2 - 1 \right\Vert_2.
		\end{align*}
		\begin{proof}
			Let us start by analyzing the spectral projection $\smash{p= E_{V_n^2}\left((0, \nicefrac{1}{4}\right))}$. Since $V_n^2$ is affiliated with $\mathcal{A}$, we have $p \in \mathcal{A}$. Due to $\smash{E_{V_n^2}(\{0 \}) = E_{V_n}(\{0\}) = E_{V_n^{\nicefrac{1}{2}}}(\{0\}) = 0}$, we can write
			\begin{align*} 
				p = E_{V_n}\left((0, \nicefrac{1}{2} )\right) = E_{V_n^{-1}}\left((2,\infty )\right) = E_{V_n^{-\nicefrac{1}{2}}}\smash{\big(\big(\sqrt{2},\infty\big)\big)}=  E_{V_n^{-2}}\left((4,\infty)\right).
			\end{align*}
			Now, define $f: \mathbb{R} \rightarrow \mathbb{R}$ by $f(x) := \vert x -1 \vert$ and note that $f^{-1}((\alpha, \infty)) = (-\infty, 1-\alpha) \cup (\alpha +1, \infty)$
		holds for any $\alpha \geq 0$.
			Thus, it follows
			\begin{align*}
				E_{\vert V_n^2 - 1 \vert}\left((\alpha, \infty)\right) &= E_{f(V_n^2)}\left((\alpha, \infty)\right) = E_{ V_n^2}((-\infty, 1-\alpha) \cup (\alpha + 1, \infty))\\ & = E_{ V_n^2}((0, 1-\alpha)) + E_{ V_n^2}((\alpha + 1, \infty))
			\end{align*}
			for $\alpha \in (0,1).$
			Combining this equation for the specific choice $\alpha = \nicefrac{3}{4}$ with Chebyshev's inequality, we obtain
			\begin{align*}
				\varphi(p) & = \varphi\left( E_{V_n^2}\left((0, \nicefrac{1}{4} )\right) \right) \leq \varphi\left( E_{V_n^2}\left((0, \nicefrac{1}{4}) \right) \right) +\varphi\left( E_{V_n^2}\left( (\nicefrac{7}{4}, \infty )\right) \right) \\ &= \varphi\left(E_{\vert V_n^2 -1 \vert}((\nicefrac{3}{4}, \infty))\right) = \mu_{\vert V_n^2 -1 \vert}((\nicefrac{3}{4}, \infty))\leq 
				2\big \Vert V_n^2 - 1 \big\Vert_2^2.
			\end{align*}
			Note that we have $\Vert V_n^2 - 1 \Vert_2 < \infty$ due to $X_i \in \mathcal{L}^4(\mathcal{A})$ for all $i \in [n]$.
			Moreover, observe that we can write $V_np = g(V_n)$ for $g: \mathbb{R} \rightarrow \mathbb{R}$ given by $g(x) := x$ for $x \in (0, \nicefrac{1}{2})$ and $g(x) := 0$ for $x \not \in (0, \nicefrac{1}{2})$. Due to $0 \leq g \leq \nicefrac{1}{2}$ and $V_n \in \Aff(\mathcal{A})_{\text{sa}}$, we get $V_n p \in \mathcal{A}$ with $0 \leq V_n p \leq \nicefrac{1}{2}$. Similarly, it follows
			\begin{align*}
				(1-p)V_n^{-\nicefrac{1}{2}} = V_n^{-\nicefrac{1}{2}}(1-p) \in \mathcal{A}, \qquad	0 \leq V_n^{-\nicefrac{1}{2}}(1-p) =  V_n^{-\nicefrac{1}{2}}E_{V_n^{-\nicefrac{1}{2}}}\big( \big(0, \sqrt{2} \,\big]\big)  \leq \sqrt{2}.
			\end{align*}

			Let us proceed by proving the claim for $\Vert S_n V_n^{-\nicefrac{1}{2}} \Vert_2.$ Together with $0 \leq S_n^2 = \vert S_n \vert^2 \leq nV_n^{2}$ following from \cref{Bikchentaev square sum smallelr sum of squares}, we obtain 
			\begin{align*}
				\big\Vert S_n V_n^{-\nicefrac{1}{2}}p \big\Vert_2^2= \varphi\left(pV_n^{-\nicefrac{1}{2}}S_n^2V_n^{-\nicefrac{1}{2}}p\right) \leq \varphi(pnV_np) = n \Vert pV_np \Vert_1  \leq n  \Vert p \Vert_1 \Vert V_np \Vert
				\leq \frac{n}{2} \varphi(p)\leq n \big\Vert V_n^2 - 1 \big\Vert_2^2.
			\end{align*}
			For later reference, note that this implies 
			\begin{align*} 
				\big(S_nV_n^{-\nicefrac{1}{2}}p\big)^*	\big(S_nV_n^{-\nicefrac{1}{2}}p\big) = pV_n^{-\nicefrac{1}{2}}S_n^2V_n^{-\nicefrac{1}{2}}p \in \mathcal{L}^1(\mathcal{A}),
			\end{align*}
			which, by the tracial property of $\varphi$ as formulated in \cref{SNS section preliminaries nc integration theory}, leads to
			\begin{align*}
				\big	\Vert pV_n^{-\nicefrac{1}{2}}S_n \big \Vert_2^2& = \varphi\left( \big( pV_n^{-\nicefrac{1}{2}}S_n\big)^*\big(pV_n^{-\nicefrac{1}{2}}S_n \big) \right) = \varphi\left(\big(S_nV_n^{-\nicefrac{1}{2}}p\big) \big(S_nV_n^{-\nicefrac{1}{2}}p\big)^*\right)\\ &= \varphi\left( \big(S_nV_n^{-\nicefrac{1}{2}}p\big)^* \big(S_nV_n^{-\nicefrac{1}{2}}p\big)\right) = \big\Vert S_n V_n^{-\nicefrac{1}{2}}p \big\Vert_2^2 \leq n \big\Vert V_n^2 - 1 \big\Vert_2^2.
			\end{align*}
			By freeness, we have $\varphi(X_i X_j) = 0$ for $i \neq j$ and thus obtain $\Vert S_n \Vert_2 = 1.$
			Combining this with 
			\begin{align*}
				\big	\Vert S_n V_n^{-\nicefrac{1}{2}} (1-p) \big\Vert_2 \leq \Vert S_n \Vert_2 \big\Vert V_n^{-\nicefrac{1}{2}}(1-p) \big\Vert \leq \sqrt{2},
			\end{align*}
			we arrive at 
			\begin{align*}
				\big\Vert S_n V_n^{-\nicefrac{1}{2}}\big \Vert_2 \leq 	\big\Vert S_n V_n^{-\nicefrac{1}{2}}p \big\Vert_2 + \big\Vert S_n V_n^{-\nicefrac{1}{2}} (1-p) \big\Vert_2  \leq  \sqrt{n} \big\Vert V_n^2 - 1 \big\Vert_2 + \sqrt{2}
			\end{align*}
			as claimed.
			
			In order to prove the remaining inequality, we use $0 \leq p =p^2 \leq 1$ as well as \cref{unbounded Un smaller sqrt n} leading to 
			\begin{align*}
				\Vert pU_np \Vert_2^2 = \varphi\big(p U_n p^2 U_n p\big) \leq \varphi\big(p U_n^2 p\big) \leq n\varphi \big( p^2 \big) \leq 2n \big\Vert V_n^2 -1 \big\Vert_2^2.
			\end{align*} 
		Together with 
			\begin{align*}
				\Vert (1-p)U_n(1-p) \Vert_2 = \big\Vert (1-p)V_n^{-\nicefrac{1}{2}}S_nV_n^{-\nicefrac{1}{2}}(1-p) \big\Vert_2 \leq \big \Vert (1-p)V_n^{-\nicefrac{1}{2}} \big \Vert \big \Vert V_n^{-\nicefrac{1}{2}}(1-p) \big \Vert \Vert S_n \Vert_2 \leq 2
			\end{align*}
			and $U_n = pU_np  +  (1-p)U_n(1-p) + (1-p)U_np + pU_n(1-p) $, we conclude
			\begin{align*}
				\Vert U_n \Vert_2 &\leq \sqrt{2n} \big\Vert V_n^2 - 1 \big\Vert_2 + 2 + \big\Vert (1-p)V_n^{-\nicefrac{1}{2}}S_nV_n^{-\nicefrac{1}{2}}p \big\Vert_2 + \big\Vert pV_n^{-\nicefrac{1}{2}}S_nV_n^{-\nicefrac{1}{2}}(1-p) \big\Vert_2 \\ 	& \leq  \sqrt{2n} \big\Vert V_n^2 - 1 \big\Vert_2 + 2 +\big \Vert (1-p)V_n^{-\nicefrac{1}{2}} \big \Vert \big\Vert S_nV_n^{-\nicefrac{1}{2}}p \big\Vert_2 +  \big \Vert pV_n^{-\nicefrac{1}{2}}S_n \big\Vert_2 	\big \Vert V_n^{-\nicefrac{1}{2}}(1-p) \big \Vert 	\\ &\leq 2 +3\sqrt{2n} \big\Vert V_n^2 - 1 \big\Vert_2.
			\end{align*}
		\end{proof}
	\end{lem}
	%Let us note that, by using the same arguments as above, we have  
	%$ \sup_{n \in \mathbb{N}} \Vert S_n V_n^{-1}  \Vert_2 \leq 2 + \sqrt{2\varphi(X^4)}.$	
	Finally, let us see how the previous lemmata help to prove \cref{unbounded_main_BE}.  
	\begin{proof}[Proof of \cref{unbounded_main_BE}]
		Let $(X_i)_{i \in \mathbb{N}}$ be given as in \cref{unbounded_main_BE}. Assume that $n \geq n_0$ holds with $n_0$ taken from \cref{unbounded_V_n is invertible}. Then, by the same lemma, the self-normalized sum $U_n$ is well-defined in $\Aff(\mathcal{A})$ proving the first claim of the theorem.
		
		In order to derive the claimed rate of convergence in terms of the Kolmogorov distance, let us -- for the moment -- additionally assume that $L_{4n} < \nicefrac{1}{16}$ holds. We will use Bai's inequality from \cref{Bai Goetze more general endl. Intgrenzen} with the choices
		\begin{align*} 
			A = 3B, \qquad	B = L_{4n}^{-\nicefrac{1}{4}} > 2, \qquad a=2,  \qquad \gamma > 0.7, \qquad	v \in (0,1), \qquad \varepsilon = 6v \in (0,2).
		\end{align*}
		The parameter $v$ will be specified in the course of this proof. Note that the integrability of $\vert F_{n} - F_\omega \vert$ is ensured by \cref{unbounded Un smaller sqrt n} and \cref{Bai Bem Wigner}. Moreover, observe that the above-made choices guarantee that the constants $\kappa$ and $C_{\gamma, \kappa}$ from \cref{Bai Goetze more general endl. Intgrenzen} do not depend on $n$.
		
		It remains to establish upper bounds for the integrals appearing on the right-hand side of the inequality in \cref{Bai Goetze more general endl. Intgrenzen}. Since the resulting bounds partly depend on $\Vert V_n^2 -1 \Vert_2$, let us note that the freeness of $X_1^2, \dots, X_n^2$ implies
		\begin{align} \label{LLN unbounded rate of convergence example}
			\big\Vert V_n^2 - 1\big\Vert_2^2
			%= \varphi\big( (V_n-^2 I)^2\right) 
			= \frac{1}{B_n^4}\varphi\left( \left( \sum_{i=1}^{n} X_i^2 - \sigma_i^2 \right)^2\right) 
			= \frac{1}{B_n^4}\sum_{i=1}^{n} \varphi\left( \big(X_i^2 - \sigma_i^2 \big)^2\right) 
			= \frac{1}{B_n^4} \sum_{i=1}^n \varphi\left(X_i^4\right) - \sigma_i^4 \leq  L_{4n}.
		\end{align}

		Now, we start with estimating one of the integrals appearing in \cref{Bai Goetze more general endl. Intgrenzen}. By Chebyshev's inequality, we obtain $1-F_n(x) \leq \Vert U_n \Vert_2^2x^{-2}$ for $x>0$ and $F_n(x) \leq   \Vert U_n \Vert_2^2x^{-2}$ for $x<0.$
		Combining this with \cref{second moment SNS} and \eqref{LLN unbounded rate of convergence example}, it follows
		\begin{align*} 
			\begin{split}
				\int_{B}^{\infty} 1 - F_n(x) dx \leq \frac{\Vert U_n \Vert_2^2}{B} \leq 8L_{4n}^{\nicefrac{1}{4}} + 36n L_{4n}^{\nicefrac{5}{4}}, \qquad \int_{-\infty}^{-B} F_n(x) dx \leq  8L_{4n}^{\nicefrac{1}{4}} + 36n L_{4n}^{\nicefrac{5}{4}}
			\end{split}
		\end{align*}
		leading to 
		\begin{align} \label{unbounded sns bai integral 1}
			\int_{\vert x \vert > B} \vert F_n(x) - F_\omega(x) \vert dx \leq 16L_{4n}^{\nicefrac{1}{4}} + 72n L_{4n}^{\nicefrac{5}{4}}.
		\end{align}
		
		Let us continue with the integrals over the difference of the Cauchy transforms. The resolvent identity implies
		\begin{align} \label{unbounded diff CT U_n and S_n}
			\begin{split}
				\vert G_n(z) - G_{S_n}(z)\vert  &=  \big\vert	\varphi\left((z-U_n)^{-1} - (z-S_n)^{-1} \right) \big\vert = \left \vert 	\varphi\left((z-U_n)^{-1}(U_n-  S_n)(z-S_n)^{-1} \right) \right\vert  \\ & \leq \frac{1}{(\Im z)^2} \Vert U_n - S_n \Vert_1, \qquad z \in \mathbb{C}^+.
			\end{split}
		\end{align}
		Rewriting
		\begin{align*}
			U_n - S_n =\left (V_n^{-\nicefrac{1}{2}} -1 \right) S_nV_n^{-\nicefrac{1}{2}} + S_n\left(V_n^{-\nicefrac{1}{2}}-1\right) = \left(1-V_n^{\nicefrac{1}{2}}\right)U_n + S_nV_n^{-\nicefrac{1}{2}}\left(1-V_n^{\nicefrac{1}{2}}\right),
		\end{align*}
	we get
		\begin{align*}
			\Vert U_n - S_n \Vert_1 \leq \left\Vert  1-V_n^{\nicefrac{1}{2}} \right\Vert _2 \left( \Vert U_n \Vert_2 + \Big\Vert S_nV_n^{-\nicefrac{1}{2}} \Big\Vert_2\right).
		\end{align*}
		By functional calculus, we deduce that $\smash{1+ V_n^{\nicefrac{1}{2}}}$ is invertible with inverse $\smash{(1+V_n^{\nicefrac{1}{2}})^{-1}\in \mathcal{A}}$ satisfying
		$\smash{\Vert (1+V_n^{\nicefrac{1}{2}})^{-1} \Vert \leq 1}$.
		%	The function $h: \mathbb{R} \rightarrow \mathbb{R}$ given by $h(x) :=  1+x^{\nicefrac{1}{4}} $ for $x \geq 0$ and $h(x) := 0$ for $x<0$
		%	is $\smash{E_{V_n^2}}$-a.s.\@ bounded from below by $1$. In particular, its reciprocal $\nicefrac{1}{h}$ with $0$ mapped to $\infty$ is $\smash{E_{V_n^2}}$-a.s.\@ bounded from above by $1$. By \cite[Theorem 5.9]{Schmuedgen2012}, the operator $\smash{1+ V_n^{\nicefrac{1}{2}}}$ is invertible with inverse $\smash{(1+V_n^{\nicefrac{1}{2}})^{-1}\in \mathcal{A}}$  satisfying
		%	$\smash{\Vert (1+V_n^{\nicefrac{1}{2}})^{-1} \Vert \leq 1}$.
		Together with $\smash{1 -V_n^{\nicefrac{1}{2}} = (1-V_n)(1+V_n^{\nicefrac{1}{2}})^{-1}}$, it follows	
		\begin{align*}
			\big\Vert 1-V_n^{\nicefrac{1}{2}} \big\Vert_2 \leq \Vert  1-V_n\Vert_2  \left\Vert \big(1+V_n^{\nicefrac{1}{2}}\big)^{-1} \right\Vert \leq  \Vert  1-V_n\Vert_2.
		\end{align*}
		Repeating the last two calculations, we arrive at $\smash{\big \Vert 1-V_n^{\nicefrac{1}{2}} \big\Vert_2 \leq  \Vert  1-V_n^2 \Vert_2.}$ With the help of \cref{second moment SNS}, we conclude
		\begin{align*}
			\Vert U_n - S_n \Vert_1 \leq L_{4n}^{\nicefrac{1}{2}} \left( \Vert U_n \Vert_2 + \left\Vert S_nV_n^{-\nicefrac{1}{2}} \right\Vert_2\right) \leq L_{4n}^{\nicefrac{1}{2}}  \left(  4 + 6\sqrt{n} L_{4n}^{\nicefrac{1}{2}} \right) .
		\end{align*}
		Integration yields
		\begin{align*}
			\int_{v}^1 \left \vert G_n(x+iy) - G_{S_n}(x+iy) \right \vert dy \leq \Vert U_n -S_n \Vert_1 \int_v^1 \frac{1}{y^2} dy \leq \frac{\Vert U_n -S_n \Vert_1 }{v} \leq v^{-1}L_{4n}^{\nicefrac{1}{2}}\left(4 +6\sqrt{n} L_{4n}^{\nicefrac{1}{2}}\right)
		\end{align*}
		for any $x \in [-2+\nicefrac{\varepsilon}{2}, 2-\nicefrac{\varepsilon}{2}]$ and 
		\begin{align*}
			\int_{-A}^{2} \vert G_n(u+i) - G_{S_n}(u+i)  \vert du 
			\leq  (2+A)\Vert U_n - S_n\Vert_1 \leq	4L_{4n}^{\nicefrac{1}{4}}  \left(  4 + 6\sqrt{n} L_{4n}^{\nicefrac{1}{2}} \right).
		\end{align*}
		
		It remains to bound the integral contributions from the difference between $G_{S_n}$ and $G_\omega$. In the proof of \cref{bounded_main_non_id}, we used integration by parts and a Berry-Esseen type estimate in order to handle these contributions. However, the same procedure would increase the final rate of convergence in the setting at hand, which is due to the fact that the corresponding Berry-Esseen rate of order $\smash{L_{3n}^{\nicefrac{1}{2}}}$ -- compare to \cref{Berry Esseen CLT unbounded} -- is too weak. Fortunately, a detailed look into the proof of \cref{Berry Esseen CLT unbounded} given in \cite[Theorem 2.6]{Chistyakov2008} solves this problem.
		According to the inequalities (6.38), (6.39), (6.43), and (6.44) in \cite{Chistyakov2008}, we have 
		\begin{align*}
			\int_{-A}^2  \vert G_{S_n}(u+i) - G_{\omega}(u+i) \vert du \leq \int_{-\infty}^\infty  \vert G_{S_n}(u+i) - G_{\omega}(u+i) \vert du \leq C_0L_{3n},
		\end{align*}
		\begin{align*}
			\sup_{x \in [-2, 2]}	\int_{v}^1 \vert G_{S_n}(x+iy) - G_{\omega}(x+iy) \vert dy \leq   C_1 \left( L_{3n}^{\nicefrac{3}{4}} + L_{3n} \vert \! \log L_{3n} \vert \right)
		\end{align*}
		for absolute constants $C_0, C_1>0$, whenever $L_{3n} \leq D_1$ and $v \geq D_2L_{3n}^{\nicefrac{1}{2}}$ are satisfied for appropriately chosen constants $D_1, D_2>0$.

		From now on, assume that the inequalities $L_{4n} < \min\{\nicefrac{1}{16}, \nicefrac{1}{(3D_2)^4}\}$, $L_{3n} < \min\{ D_1, 1\}$ hold. We set  $\smash{v: = D_2L_{4n}^{\nicefrac{1}{4}}}$ and note that we have $\smash{v \geq D_2L_{3n}^{\nicefrac{1}{2}}}$. We obtain
		\begin{align*}
			\int_{-A}^2  \vert G_{n}(u+i) - G_{\omega}(u+i) \vert du \leq 16L_{4n}^{\nicefrac{1}{4}} + 24 \sqrt{n}L_{4n}^{\nicefrac{3}{4}} + C_0L_{4n}^{\nicefrac{1}{2}} 
		\end{align*}
		and 
		\begin{align*} 
			\sup_{x \in [-2+\nicefrac{\varepsilon}{2}, 2-\nicefrac{\varepsilon}{2}]} \int_{v}^1 \vert G_{n}(x+iy) - G_{\omega}(x+iy) \vert dy \leq  D_2^{-1}\left( 4L_{4n}^{\nicefrac{1}{4}} + 6\sqrt{n}L_{4n}^{\nicefrac{3}{4}}\right) + 2C_1L_{4n}^{\nicefrac{1}{4}}.
		\end{align*}
		Combining the last two estimates with \eqref{unbounded sns bai integral 1}, \cref{Bai Goetze more general endl. Intgrenzen}, and \cref{Bai Bem Wigner}, it follows
		\begin{align*}
			\Delta(\mu_n, \omega) \leq C_3\left(L_{4n}^{\nicefrac{1}{4}} + \sqrt{n}L_{4n}^{\nicefrac{3}{4}} + nL_{4n}^{\nicefrac{5}{4}}\right)
		\end{align*}
		for some constant $C_3>0$ (independent of $n$), whenever the conditions on $L_{4n}$ and $L_{3n}$ given above are satisfied.
		
		Lastly, note that if at least one of these conditions does not hold, we have $\smash{L_{4n}^{\nicefrac{1}{4}} + \sqrt{n}L_{4n}^{\nicefrac{3}{4}} + nL_{4n}^{\nicefrac{5}{4}}> C_4}$
		for some $C_4>0$ by using the inequality $L_{4n} \geq L_{3n}^2.$ Consequently, we get
		\begin{align*}
			\Delta(\mu_n, \omega) \leq \max\left\{C_3, C_4^{-1}\right\}\left(L_{4n}^{\nicefrac{1}{4}} + \sqrt{n}L_{4n}^{\nicefrac{3}{4}} + nL_{4n}^{\nicefrac{5}{4}}\right)
		\end{align*}
		for all $n \geq n_0.$
		
		In order to prove the last claim, assume that $\lim_{n \rightarrow \infty} \sqrt{n}L_{4n} = 0$ holds. Due to \eqref{unbounded diff CT U_n and S_n}, we obtain
		\begin{align*}
			\vert G_n(z) - G_{S_n}(z) \vert \leq \frac{1}{(\Im z)^2} L_{4n}^{\nicefrac{1}{2}}  \left(  4 + 6\sqrt{n} L_{4n}^{\nicefrac{1}{2}} \right) \rightarrow 0
		\end{align*}
	as $n \rightarrow \infty$ for all $z \in \mathbb{C}^+$.  Together with $\lim_{n \rightarrow \infty} G_{S_n}(z) = G_\omega(z)$ for all $z \in \mathbb{C}^+$ following from the free Lindeberg CLT, the claimed weak convergence of $\mu_n$ to $\omega$ as $n \rightarrow \infty$ is immediate.
	\end{proof}	
	
	Let us continue with the proof of \cref{unbounded main id} considering the special case of identical distributions.
	\begin{proof}[Proof of \cref{unbounded main id}]
		Let $(X_i)_{i \in \mathbb{N}}$ be given as in \cref{unbounded main id}. Keeping the notation introduced in \cref{unbounded_main_BE}, we have $B_n^2 = n > 0$, $L_{4n} = \Vert X_1 \Vert_4^4 n^{-1}$. Moreover, Lindeberg's condition is satisfied. By the last-mentioned theorem, we find $n_0 \in \mathbb{N}$ such that the self-normalized sum $U_n$ exists in $\Aff(\mathcal{A})$ for $n \geq n_0$ and its analytic distribution $\mu_n$ satisfies 
		\begin{align*}
			\Delta(\mu_n, \omega) \leq C_0\frac{\Vert X_1\Vert_4 + \Vert X_1 \Vert_4^3 + \Vert X_1 \Vert_4^5}{n^{\nicefrac{1}{4}}} \leq 3C_0 \frac{\Vert X_1 \Vert_4^5}{n^{\nicefrac{1}{4}}}
		\end{align*} 
		for all $n \geq n_0$ with $C_0>0$ being the constant from \cref{unbounded_main_BE}. The weak convergence of $\mu_n$ to $\omega$ as $n \rightarrow \infty$ follows immediately from $\lim_{n \rightarrow \infty} \Delta(\mu_n, \omega) = 0.$
	\end{proof}

	We end the discussion of unbounded self-normalized sums with two comments.
	\begin{rem} \label{SNS unbounded Endbemerkung}
		\begin{enumerate}[(i)]
			\item  In view of the rates of convergence established in the context of the free CLT in \cref{Berry Esseen CLT unbounded}, the rates obtained in \cref{unbounded_main_BE} and \cref{unbounded main id} do not seem to be optimal. The reason why we were not able to go beyond these non-optimal rates in our proofs is twofold: First, the bound on the difference between the Cauchy transforms $G_n(z)$ and $G_{S_n}(z)$ in \eqref{unbounded diff CT U_n and S_n} is too weak -- in particular due to the factor $(\Im z)^{-2}$. However, any attempt to improve this bound (for instance by proceeding as in the proof of \cref{bounded_main_non_id}) was not successful since evaluations of the (extended) functional $\varphi$ in unbounded random variables are rather delicate to handle. Second, the support of $\mu_n$ a priori grows with $n$, which forces us to choose the parameters $A$ and $B$ in \cref{Bai Goetze more general endl. Intgrenzen} dependently on $n$. 
			\item  In \cref{SNS bounded Remark Yin}, we saw that the approach via Cauchy transforms  can be substituted by a convergence result for rational expressions in strongly convergent bounded random variables. Recently, Collins et al.\@ \cite[Proposition 28]{Collins2022} established a similar result in the unbounded setting. However, the assumptions made in the last-cited proposition are stronger than the ones needed in \cref{unbounded_main_BE}: Among others, we would have to assume that the underlying sequence of random variables is contained in $\mathcal{L}^k(\mathcal{A})$ for any $1 \leq k < \infty$ in order to be able to derive weak convergence of the analytic distribution of the self-normalized sum to Wigner's semicircle law.
		\end{enumerate}
	\end{rem}

\end{document}